\documentclass[11pt]{article}
\usepackage{stmaryrd}
\usepackage{amsmath}
\usepackage{amsfonts}
\usepackage{amssymb}
\usepackage{amsthm}
\usepackage[all,2cell]{xy}
\usepackage{color}
\usepackage{etoolbox}

\usepackage[T2A,T1]{fontenc}
\usepackage[utf8]{inputenc}
\usepackage[russian,USenglish]{babel}

\UseTwocells

\newbool{hidedetails}
\booltrue{hidedetails}

\newtheorem{lemma}{Lemma}
\newtheorem{proposition}{Proposition}
\newtheorem{corollary}{Corollary}
\newtheorem{theorem}{Theorem}

\theoremstyle{definition}
\newtheorem{definition}{Definition}
\newtheorem{example}{Example}

\newtheorem{remark}{Remark}

\newcommand{\sets}{{\rm Set}}											
\newcommand{\set}{{S}}												
\newcommand{\element}{{s}}											
\newcommand{\ordina}{{\alpha}}										
\newcommand{\finor}[1]{{\underline{#1}}}									
\newcommand{\colim}{{\rm colim}}										
\newcommand{\id}{{\rm Id}}											
\newcommand{\cardinal}{{\aleph}}										
\newcommand{\cardint}[1]{{\underline{#1}}}								
\newcommand{\poset}{{L}}											
\newcommand{\cardi}[1]{{|#1|}}											
\newcommand{\siqu}{{\sigma}}											

\newcommand{\abel}{{\rm Ab}}											
\newcommand{\yoneda}[2]{{\hom_{#2}(#1,-)}}								
\newcommand{\caty}{{\mathcal C}}										

\newcommand{\adica}{{\mathcal C}}										
\newcommand{\adifu}{{\mathcal F}}										
\newcommand{\mor}{{f}}												
\newcommand{\cokernel}[1]{{{\rm Coker}(#1)}}								
\newcommand{\image}{{\rm Im}}										
\newcommand{\coimage}{{\rm Coim}}									
\newcommand{\K}{{\rm K}}											
\newcommand{\obj}{{C}}												
\newcommand{\pobj}{{P}}												
\newcommand{\qobj}{{Q}}												
\newcommand{\emor}{{p}}												
\newcommand{\momor}{{i}}											

\newcommand{\quabi}{{\mathcal Q}}										
\newcommand{\abi}{{\mathcal  A}}										
\newcommand{\abion}[1]{{\mathbb A(#1)}}									
\newcommand{\abiono}{{\mathbb A}}										

\newcommand{\gaugenorm}[1]{{\gamma_{#1}}}								
\newcommand{\linearspan}[2]{{\langle #1\rangle}}							
\newcommand{\disk}{{D}}												
\newcommand{\vectorspace}{{V}}										
\newcommand{\banachspace}{{B}}										
\newcommand{\bornospace}{{\mathcal B}}								
\newcommand{\bornologybase}{{\mathfrak D}}								
\newcommand{\boudis}[1]{{\mathfrak D_{#1}}}						
\newcommand{\bornomor}{{\phi}}										
\newcommand{\coconbo}{{\rm CBorn}}									
\newcommand{\banco}{{\rm Ban}}										
\newcommand{\monobor}{{\coconbo_\mu}}								
\newcommand{\ctensor}{{\,\widehat{\otimes}\,}}								
\newcommand{\norm}[2]{{\|#1\|_{#2}}}									
\newcommand{\summable}[1]{{l^1(#1)}}									
\newcommand{\indo}[1]{{\rm Ind(#1)}}									
\newcommand{\internalclosure}[1]{{\overline{#1}}}							
\newcommand{\avector}{{v}}											
\newcommand{\dechar}[1]{{\Xi_{#1}}}									
\newcommand{\spherun}[1]{{{\rm S}_{#1}}}								
\newcommand{\recoco}[1]{{{\rm CBorn}^{<#1}}}								
\newcommand{\coubo}{{\underline{\rm CBorn}}}							
\newcommand{\samcou}{{l^1(\mathbb N)}}								
\newcommand{\resize}[2]{{\mathcal R_{#2}(#1)}}							
\newcommand{\Hom}{{\rm Hom}}										

\newcommand{\wball}[4]{{\rm B_{#1,#2}^{#3,#4}}}							
\newcommand{\jet}[2]{{{\rm J}^{#2}(#1)}}									
\newcommand{\mindex}[1]{{\underline{#1}}}								
\newcommand{\diffo}[2]{{D^{#2}(#1)}}									
\newcommand{\cpt}[1]{{{\rm Cpt}(#1)}}									
\newcommand{\cpto}{{\rm Cpt}}											
\newcommand{\fpace}{{\mathcal F}}										
\newcommand{\fcati}{{\mathfrak F}}										
\newcommand{\qcoh}[1]{{{\rm QCoh}(#1)}}								

\newcommand{\monoca}{{\mathcal C}}									
\newcommand{\monone}{{\mathbf 1}}									
\newcommand{\monote}{{\bullet}}										
\newcommand{\monoid}{{\rm A}}										
\newcommand{\monopro}[1]{{\mu_{#1}}}									
\newcommand{\monou}[1]{{\iota_{#1}}}									
\newcommand{\ucori}[1]{{\rm Comm({#1})}}								
\newcommand{\freecou}{{\rm F}}										
\newcommand{\forget}{{\rm U}}											

\newcommand{\simpo}[1]{{{\rm S}#1}}										
\newcommand{\simo}[1]{{#1_{\bullet}}}									
\newcommand{\ssets}{{\rm SSet}}										
\newcommand{\sset}{{\mathbb S}}										
\newcommand{\simplex}[1]{{\Delta[#1]}}									
\newcommand{\boundary}[1]{{\partial\Delta[#1]}}							
\newcommand{\boun}[1]{{\partial_{#1}}}									
\newcommand{\degen}[1]{{s^{#1}}}										
\newcommand{\boundaryo}{{\partial}}									
\newcommand{\horn}[2]{{\Lambda[#1,#2]}}								
\newcommand{\gencof}{{I}}											
\newcommand{\gentrico}{{J}}											

\newcommand{\compo}[1]{{\rm K#1}}									
\newcommand{\como}[1]{{#1^\bullet}}									
\newcommand{\normach}[1]{{{\mathcal  N}(#1)}}							
\newcommand{\normacho}{{\mathcal N}}									
\newcommand{\simcho}[1]{{\Gamma(#1)}}									
\newcommand{\simchoo}{{\Gamma}}										

\newcommand{\cinfty}{{\mathcal C^\infty}}									
\newcommand{\allrings}{{\rm C^\infty R}}									
\newcommand{\cring}{{A}}												
\newcommand{\cmor}{{\phi}}											
\newcommand{\kernel}[1]{{{\rm Ker}(#1)}}									
\newcommand{\frifig}{{\rm FC^\infty R_{f.g.}}}								
\newcommand{\freci}{{\rm FC^\infty R}}									
\newcommand{\alfreci}{{\widetilde{\rm SC^\infty R}}}							
\newcommand{\resolution}[1]{{\widetilde{#1}}}								
\newcommand{\manifold}{{\mathcal X}}									

\newcommand{\modules}[1]{{{\rm Mod}({#1})}}								
\newcommand{\module}{{M}}											
\newcommand{\dqcoh}[1]{{{\rm DQcoh}(#1)}}								

\newcommand{\ideal}{{\mathfrak A}}						

\newcommand{\rpt}{{\rm p}}							
\newcommand{\rqt}{{\rm q}}							

\newcommand{\blue}{\color{blue}}

\newcommand{\hide}[1]{\ifbool{hidedetails}{}{{\blue #1}}}

\begin{document}

\title{Quasi-coherent sheaves in differential geometry}
\author{\small\parbox{.5\linewidth}{Dennis Borisov\\ University of Goettingen, Germany\\ Bunsenstr.\@ 3-4, 37073 G\"ottingen\\ dennis.borisov@gmail.com}
\parbox{.5\linewidth}{Kobi Kremnizer\\ University of Oxford, UK\\ Woodstock Rd, Oxford OX2 6GG\\ yakov.kremnitzer@maths.ox.ac.uk}}
\date{\today}
\maketitle

\begin{abstract} It is proved that the category of simplicial complete bornological spaces over $\mathbb R$ carries a combinatorial monoidal model structure satisfying the monoid axiom. For any commutative monoid in this category the category of modules is also a monoidal model category with all cofibrant objects being flat. In particular, weak equivalences between these monoids induce Quillen equivalences between the corresponding categories of modules. On the other hand, it is also proved that the functor of pre-compact bornology applied to simplicial $\cinfty$-rings preserves and reflects weak equivalences, thus assigning stable model categories of modules to simplicial $\cinfty$-rings.

\smallskip



\end{abstract}

\tableofcontents

\section{Introduction}

Our goal (beyond this paper) is to adapt some very useful techniques from algebraic geometry (e.g.\@ \cite{BD04}) to differential geometry. Among other things we are particularly interested in constructing chiral algebras using factorization spaces within $\cinfty$-context (e.g.\@ \cite{FB-Z04} \S20.3). An essential requirement for this is having a good theory of quasi-coherent and coherent sheaves. Here ``good'' means having the usual functoriality: inverse and direct images with or without proper support, base change formulas, etc.

Our goal in this paper is to provide the quasi-coherent part of the theory. One might ask why would sheaves appear in the $\cinfty$-setting, where the functors of global sections are exact. They appear because chiral algebras come from very large geometric objects, that are $Ind$-schemes at best. 

Good functoriality of the theory of modules is indispensable when dealing with such objects. In other words, while for a $\cinfty$-scheme a sheaf of modules is given by its global sections, in general our sheaves live on a site of $\cinfty$-schemes. From this point of view our construction is not a definition but a theorem: we show that with the right notion of modules over $\cinfty$-rings one obtains the expected behaviour on a global site.\footnote{We have the Zariski topology in mind, but others might work as well.}

\smallskip

Our plan then is to study the geometry of $\cinfty$-rings, merging algebraic geometry with techniques of functional analysis. The analytical techniques are crucial, since $\cinfty$-rings are much more than just commutative algebras. For example, modules over $\cinfty$-rings without additional structure are not very useful: even for a vector bundle $\module$ of finite rank the canonical base change morphism $\Phi''^*(\Phi'_*(\module))\rightarrow\pi''_*(\pi'^*(\module))$ in a cartesian diagram
	\begin{equation*}\xymatrix{\manifold'\underset{\manifold}\times\manifold''\ar[rr]^{\pi''}\ar[d]_{\pi'} && \manifold''\ar[d]^{\Phi''}\\
	\manifold'\ar[rr]_{\Phi'} && \manifold}\end{equation*}
is usually not an isomorphism. The solution is well known: view $\cinfty$-rings as commutative monoids in a closed monoidal category coming from analysis and take the corresponding categories of modules (cf.\@ \cite{BK15} for an overview).

Following \cite{KM97}, \cite{WN05}, \cite{Meyer2007} and others we choose this monoidal category to be the category $\coconbo$ of complete convex bornological $\mathbb R$-spaces, together with the completed projective tensor product. It is important for these spaces and for the tensor product to be completed to provide a meaningful commutative algebra of $\cinfty$-rings (cf.\@ \cite{Meyer08} for a similar effect). However, this also creates a problem: some very important $\cinfty$-rings are not complete as bornological spaces. For example rings of germs are usually not complete. Another example is $\cinfty(\mathbb R)/(e^{-\frac{1}{x^2}})$, such $\cinfty$-rings substitute Artinian rings in the theory of $\infty$-nilpotent elements in $\cinfty$-algebra (\cite{BK2017}).

\smallskip

Of course any $\cinfty$-ring that is not complete as a bornological space can be given a resolution in terms of bornologically complete $\cinfty$-rings (e.g.\@ the free resolution). Thus we are led to work with simplicial objects in $\coconbo$. If we were interested only in the linear constructions, we could have used the derived category corresponding to the quasi-abelian structure on $\coconbo$ (\cite{Schneiders99}, \cite{ProsmansSchneiders00}), whose left heart is the well known category of bornological quotients (e.g.\@ \cite{WN05}). Since we would like to do commutative algebra we need to have more control, and need to develop homotopical algebra in $\coconbo$. Thus, even if we want just the categories of quasi-coherent sheaves in $\cinfty$-geometry, we need to start with derived geometry.

The weak equivalences in this derived geometry are given by evaluating morphisms at projective objects of $\coconbo$, i.e.\@ the spaces of sequences $\summable{\set}$. This is the well known construction of a model structure on a category of simplicial objects (\cite{Qu67}). Truncations of these evaluation functors are the ``miracle functors'' in the theory of quotient spaces (e.g.\@ \cite{WN05}).

\smallskip

To have a meaningful commutative algebra of $\cinfty$-rings in this setting we need to show that the category $\simpo{\ucori{\coconbo}}$ of simplicial commutative monoids in $\coconbo$ also carries a model structure, for each such monoid the category of modules inherits a model structure, and two weakly equivalent monoids have Quillen equivalent categories of modules. Then we need to show that the category of simplicial $\cinfty$-rings embeds into $\simpo{\ucori{\coconbo}}$, thus equipping every simplicial $\cinfty$-ring with a category of modules, uniquely defined up to a Quillen equivalence.

To make all this possible we need to take seriously the size of the Banach spaces involved. A complete bornological space $\bornospace$ is a vector space that is locally Banach. These Banach pieces supplement the mere vector space structure on $\bornospace$ with some convergence processes. How much convergence is allowed depends on the density characters of Banach pieces. 

For example, if we allow only finite dimensional pieces, $\bornospace$ is just a vector space. The natural bornology on $\cinfty$-rings is locally separable, i.e.\@ we also allow quotients of $\samcou$. They let us compute infinite sums, leading to completed tensor products and meaningful commutative algebra of $\cinfty$-rings. 

\smallskip

Having obtained the category of modules for each simplicial $\cinfty$-ring $\simo{\cring}$, we define the corresponding derived category of quasi-coherent sheaves $\dqcoh{\simo{\cring}}$ to be a stabilization of the model category of modules. We use the stabilized model categories of modules because we would also like to have a meaningful theory of {\it coherent sheaves}, which we define as objects in $\dqcoh{\simo{\cring}}$ homotopically of finite presentation. This forces us to view sub-objects of Banach bundles as quotients, i.e.\@ to stabilize.
	
In order to keep this paper within reasonable size we postpone the analysis of stabilization and functorial properties of coherent sheaves to a future work, but we look at the heart of the standard $t$-structure on $\dqcoh{\cring}$ in case $\cring$ is a Fr\'echet $\cinfty$-ring and of finite presentation (e.g.\@ the ring of functions on a manifold). As one would expect after \cite{Schneiders99}, it is just the left abelian envelope of the category of complete bornological $\cring$-modules.

\smallskip

We should mention that the problem of defining quasi-coherent sheaves and quasi-coherent $D$-modules in $\cinfty$-geometry can be addressed also in another way. One can  consider diffeological and eventually differentiable vector spaces (\cite{KKM87}, \cite{CostelloGwilliam}). In this approach modules are seen as sheaves on a site of $\cinfty$-rings together with additional structure (linearity, connection, etc.). 

\medskip

Here is the plan of the paper:

\smallskip

In section \ref{QuasiAbelianStructure} we recall the basic facts concerning complete bornological spaces over $\mathbb R$ and the quasi-abelian structure on their category. In section \ref{ProjectiveObjectsMonoidalStructure} we recall the properties of projective Banach and bornological spaces, and the completed projective tensor product.

In section \ref{LocallySeparableSpaces} we introduce the filtration on the category of bornological spaces according to the density characters of their Banach pieces, and in section \ref{LocalPresentabilitySmallObjects} we analyze the presentability properties of bornological spaces and show that every component of the filtration from Section \ref{LocallySeparableSpaces} is a locally presentable category.

In section \ref{ModelStructures} we recall the basics of constructing model structures using projective classes and show that the categories of bornological spaces with bounds on cardinalities are combinatorial. Moreover, we show that the completed projective tensor product satisfies the axioms of monoidal model categories, giving us model structures on the categories of modules over commutative monoids. We also look at how adjunctions between quasi-abelian categories induce Quillen adjunctions and equivalences between the corresponding categories of simplicial objects.

In section \ref{CofibrantResolutionsFlatness} we show that under some assumptions on projective objects in a quasi-abelian category (satisfied by $\coconbo$) every cofibrant module over a commutative monoid is flat, and moreover tensoring with flat modules preserves weak equivalences. This implies that the categories of modules are well defined on weak equivalence classes of commutative monoids.

In section \ref{PrecompactBornology} we extend the pre-compact bornology construction from finitely generated $\cinfty$-rings to arbitrary freely generated $\cinfty$-rings, s.t.\@ the resulting functor is fully faithful and preserves finite coproducts. In section \ref{ModelStructureFunctions} we show that $\cpto$ preserves and reflects weak equivalences. Finally in section \ref{QuasiCoherentSheaves} we use $\cpto$ to define categories of quasi-coherent sheaves, show that the base change morphism is a natural weak equivalence and consider the example of Fr\'echet $\cinfty$-rings of finite presentation.

\bigskip

{\Large Acknowledgements}

\smallskip

Part of this work was done when the first author was visiting Max-Planck Institut f\"ur Mathematik in Bonn, several short visits to Oxford were also very helpful. Hospitality and financial support from both institutions are greatly appreciated. This work was presented in a series of talks in G\"ottingen, and the first author is grateful for the comments from C.Zhu and other participants. Very helpful comments also came from D.Christensen and R.Meyer.

\section{The category of complete bornological $\mathbb R$-spaces}

In this section we assemble the necessary tools to deal with homotopical algebra of commutative bornological rings and their modules. Much of what we have collected here is very well known. To a certain extent we follow \cite{OrenKobi}.

\begin{definition} Let $\vectorspace$ be an $\mathbb R$-vector space, possibly infinite dimensional.\begin{itemize}
\item[1.] (E.g.\@ \cite{Hogbe-Nlend1977} \S0.A.3) {\it A disk} in $\vectorspace$ is a subset $\disk\subseteq\vectorspace$, s.t.\@ $\disk$ is {\it convex}, i.e.\@ $t b_0+(1-t)b_1\in \disk$ for all $b_0,b_1\in\disk$, $t\in[0,1]$, and $\disk$ is {\it circled}, i.e.\@ $t\disk\subseteq\disk$ for all $t\in[-1,1]$.
\item[2.] (E.g.\@ \cite{Hogbe-Nlend1977} \S0.A.5) The {\it semi-normed subspace generated by a disk} $\disk\subseteq\vectorspace$ is the linear span $\linearspan{\disk}{\vectorspace}:=\mathbb R\cdot\disk$ equipped with the {\it gauge semi-norm} 
	\begin{equation*}\gaugenorm{\disk}(v):=\inf\,\{r\in\mathbb R_{>0}\,|\, v\in r\disk\}.\end{equation*}
\item[3.] (E.g.\@ \cite{Hogbe-Nlend1977} \S III.3.2, \S I.1.1) {\it A base of a complete convex bornology} on $\vectorspace$ is a set of disks $\bornologybase=\{\disk_i\}_{i\in I}$ in $\vectorspace$, s.t.\@ $\vectorspace=\underset{i\in I}\bigcup\disk_i$; $\disk',\disk''\in\bornologybase\Rightarrow\disk'+\disk''\in\bornologybase$; $\disk\in\bornologybase\text{ and }r\in\mathbb R_{>0}\Rightarrow r\cdot\disk\in\bornologybase$; $\forall\disk\in\bornologybase$ the semi-normed space $(\linearspan{\disk}{\vectorspace},\gaugenorm{\disk})$ is a Banach space. {\it A complete convex bornology} on $\vectorspace$ is given by an equivalence class of bases, where $\bornologybase'\sim\bornologybase''$, if $\forall\disk'\in\bornologybase'$ $\exists\disk''\in\bornologybase''$, s.t.\@ $\disk'\subseteq\disk''$, and $\forall\disk''\in\bornologybase''\,\exists\disk'\in\bornologybase'\;\disk''\subseteq\disk'$. {\it A subset $S\subseteq\vectorspace$ is bounded} with respect to $\bornologybase$, if $\exists\disk\in\bornologybase$ such that $S\subseteq\disk$.
\item[4.] (E.g.\@ \cite{Hogbe-Nlend1977} \S I.1.2) {\it A morphism between complete convex bornological spaces} is a morphism of $\mathbb R$-vector spaces $\vectorspace\rightarrow\vectorspace'$, that maps every bounded subset of $\vectorspace$ to a bounded subset of $\vectorspace'$. We will denote the category of complete convex bornological spaces by $\coconbo$.
\end{itemize}\end{definition}
In \cite{Meyer2007} \S1.1.1 there is an additional condition for a subset to be a disk: $\disk$ has to be {\it internally closed}, meaning that $\forall r\in[0,1)$ $r v\in\disk\Rightarrow v\in\disk$. The operation of internal closure $\internalclosure{\disk}:=\{\avector\in\vectorspace\,|\,\forall r\in[0,1)\,r\avector\in\disk\}$ is idempotent, \hide{indeed let $\avector\in\internalclosure{\internalclosure{\disk}}$, $r\in [0,1)$ we can find $r',r''\in[0,1)$ s.t.\@ $r=r' r''$, then by definition $r''\avector\in\internalclosure{\disk}$ and $r\avector=r' r''\avector\in\disk$, hence $\avector\in\internalclosure{\disk}$, }%
$\internalclosure{\disk}$ is internally closed, \hide{(let $\avector\in\vectorspace$ be s.t.\@ $\forall r\in[0,1)\, r\avector\in\internalclosure{\disk}$ and let $r', r''\in[0,1)$, then $r''\avector\in\internalclosure{\disk}$ and hence $r' r''\avector\in\disk$; since every $r\in[0,1)$ can be written as $r' r''$, we conclude that $\avector\in\internalclosure{\disk}$)
}%
$\linearspan{\disk}{\vectorspace}=\linearspan{\internalclosure{\disk}}{\vectorspace}$, \hide{indeed $\avector\in\internalclosure{\disk}\Rightarrow\frac{1}{2}\avector\in\disk$, i.e.\@ $\internalclosure{\disk}\subseteq\mathbb R_{>0}\cdot\disk$,
}%
and the gauge norms on $\linearspan{\disk}{\vectorspace}$ defined by $\disk$ and $\internalclosure{\disk}$ coincide, \hide{indeed let $\avector\in\linearspan{\disk}{\vectorspace}=\linearspan{\internalclosure{\disk}}{\vectorspace}$, since $\disk\subseteq\internalclosure{\disk}$ it is clear that $\gaugenorm{\disk}(\avector)\geq\gaugenorm{\internalclosure{\disk}}(\avector)$; suppose $\gaugenorm{\disk}(\avector)>\gaugenorm{\internalclosure{\disk}}(\avector)$, in particular $\gaugenorm{\disk}(\avector)>0$, let $\avector':=\frac{1}{\gaugenorm{\disk}(\avector)}\avector$, clearly $\gaugenorm{\disk}(\avector')=1$, i.e.\@ $\forall r\in[0,1)\,r\avector'\in\disk$ hence $\avector'\in\internalclosure{\disk}$ and $\gaugenorm{\internalclosure{\disk}}(\avector')=1$ implying $\gaugenorm{\internalclosure{\disk}}(\avector)=\gaugenorm{\disk}(\avector)$,
}%
therefore the notion of closed convex bornology defined by disks coincides with the one defined by internally closed disks.

\begin{example} Let $(\banachspace,\norm{-}{\banachspace})$ be a Banach space, the disks $\{\disk_r:=\{b\,|\,\norm{b}{\banachspace}\leq r\}\}_{r\in\mathbb R_{>0}}$ define a complete convex bornology on $\banachspace$. Bounded maps between Banach spaces coincide with morphisms between the corresponding complete bornological spaces. Thus the category $\banco$ of Banach spaces and bounded maps is a full subcategory of $\coconbo$.\end{example}

When discussing sub-objects in $\coconbo$ we will need the following notion.

\begin{definition}\label{BornologicalTopology} (E.g.\@ \cite{Meyer2007} Def.\@ 1.35, 1.37) Let $\bornospace\in\coconbo$, a sequence $\{x_k\}_{k\in\mathbb N}$ in $\bornospace$ is {\it bornologically convergent}, if there is $\disk\in\bornologybase$ s.t.\@ $\{x_k\}_{k\in\mathbb N}\subseteq\linearspan{\disk}{\bornospace}$ and converges with respect to the gauge norm. A subset $\set\subseteq\bornospace$ is {\it bornologically closed}, if the limit of every convergent sequence in $\set$ also belongs to $\set$.\end{definition}

\subsection{Quasi-abelian structure}\label{QuasiAbelianStructure}

Recall (\cite{Qu67} \S 4) that in a finitely complete category a morphism $X\rightarrow Y$ is {\it an effective epimorphism}, if it is a co-equalizer of $X\underset{Y}\times X\rightrightarrows X$. In a finitely complete and cocomplete category effective epimorphisms coincide with the co-equalizers (e.g.\@ \cite{KashiwaraSchapira06} Prop.\@ 5.1.5). In an additive category coequalizers are the same as cokernels, and they are often called {\it strict epimorphisms}. Dually kernels are called {\it strict monomorphisms} (e.g.\@ \cite{Schneiders99} Rem.\@ 1.1.2).

\begin{definition} (E.g.\@ \cite{Schneiders99}, Def.\@ 1.1.3) A finitely complete and cocomplete additive category is called {\it quasi-abelian}, if pullbacks (pushouts) of strict epimorphisms (strict monomorphisms) are morphisms of the same kind.\end{definition}
Both $\banco$ and $\coconbo$ are quasi-abelian (e.g.\@ \cite{Prosmans95} Prop.\@ 3.1.7 and \cite{ProsmansSchneiders00} Prop.\@ 5.6). A morphism $\bornomor\colon\bornospace\rightarrow\bornospace'$ in $\coconbo$ is a strict epimorphism, if and only if $\bornomor$ is surjective and for every bounded $\disk'\subseteq\bornospace'$ there is a bounded $\disk\subseteq\bornospace$, s.t.\@ $\disk'\subseteq\bornomor(\disk)$ (e.g.\@ \cite{Meyer2007} Def.\@ 1.59, Lemma 1.61).

\begin{proposition}\label{StrictMonomorphisms} (E.g.\@ \cite{ProsmansSchneiders00} Prop.\@ 5.6, Cor.\@ 4.7) A morphism $\bornospace\rightarrow\bornospace'$ in $\coconbo$ is a strict monomorphism, iff it is injective, its image is closed (Def.\@ \ref{BornologicalTopology}), and the bornology on $\bornospace$ is induced by the bornology on $\bornospace'$.\end{proposition}

A morphism in a quasi-abelian category, that is simultaneously a strict epimorphism and a strict monomorphism, has to be an isomorphism (e.g.\@ \cite{Schneiders99} Prop.\@ 1.1.4). This is not true without the strictness assumption: any injective morphism between Banach spaces, that has a dense image, is simultaneously a monomorphism and an epimorphism.

\begin{definition} In a quasi-abelian category $\quabi$ a morphism $\mor\colon\qobj_1\rightarrow\qobj_2$ is {\it strict}, if $\mor=\momor\circ\emor$, where $\emor$ is a strict epimorphism and $\momor$ is a strict monomorphism (e.g.\@ \cite{Schneiders99} Rem.\@ 1.1.2(c)). A null sequence $\qobj_1\overset{\mor_1}\longrightarrow\qobj_2\overset{\mor_2}\longrightarrow\qobj_3$ is {\it strictly exact}, if $\qobj_1\rightarrow\kernel{\mor_2}$ is a strict epimorphism (e.g.\@ \cite{Schneiders99} Def.\@ 1.1.9). \hide{
In the loc.\@ cit.\@ such a sequence is called strictly exact, if $\mor_1$ is strict and $\image(\mor_1)\rightarrow\kernel{\mor_2}$ is an isomorphism. Since $\kernel{\mor_2}\rightarrow\qobj_2$ is a strict monomorphism, also in our definition $\mor_1$ has to be strict. By definition (loc.\@ cit.\@ Def.\@ 1.1.1) $\image(\mor_1)$ is the kernel of $\qobj_2\rightarrow\cokernel{\mor_1}$, therefore, if $\qobj_1\rightarrow\kernel{\mor_2}$ is a strict epimorphism, so is $\image(\mor_1)\rightarrow\kernel{\mor_2}$ (e.g.\@ loc.\@ cit.\@ Prop.\@ 1.1.8). In the same way from $\image(\mor_1)\rightarrow\qobj_2$ being a strict monomorphism, it follows that $\image(\mor_1)\rightarrow\kernel{\mor_2}$ is a strict monomorphism. Thus we conclude that $\image(\mor_1)\rightarrow\kernel{\mor_2}$ has to be an isomorphism also according to our definition. Conversely, suppose that the sequence is strictly exact according to loc.\@ cit.  As $\mor_1$ is strict, $\coimage(\mor_1)\rightarrow\image(\mor_1)$ is an isomorphism, hence $\qobj_1\rightarrow\image(\mor_1)$ is a strict epimorphism (loc.\@ cit.\@ Prop.\@ 1.1.4). Composing with the isomorphism $\image(\mor_1)\rightarrow\kernel{\mor_2}$ we obtain our definition. }%
A null sequence $\ldots\rightarrow\qobj\rightarrow\qobj'\rightarrow\ldots$ of any length is {\it strictly exact}, if every consecutive pair of morphisms is strictly exact. \end{definition}
In $\banco$ a morphism is strict, if and only if its image is closed. This is true more generally for all Fr\'echet spaces (\cite{BourTopSpaces1-5}, Thm.\@ IV.2.1). In particular a morphism in $\banco$ is a strict epimorphism, if and only if it is surjective. 

\begin{example} Let $\banachspace\rightarrow\banachspace'$ be an injective but not surjective morphism of Banach spaces, s.t.\@ the image is dense in $\banachspace'$. Clearly this morphism is not strict, however $0\rightarrow\banachspace\rightarrow\banachspace'$ is strictly exact.\end{example}

\begin{definition} (E.g.\@ \cite{Schneiders99} Def.\@ 1.1.18) A functor $\adifu\colon\quabi\rightarrow\quabi'$ between quasi-abelian categories is {\it strictly exact}, if for every strictly exact sequence $\qobj_1\rightarrow\qobj_2\rightarrow\qobj_3$, the sequence $\adifu(\qobj_1)\rightarrow\adifu(\qobj_2)\rightarrow\adifu(\qobj_3)$ is strictly exact in $\quabi'$.\end{definition}
A morphism $\qobj_1\rightarrow\qobj_2$ is a strict epimorphism, iff $\qobj_1\rightarrow\qobj_2\rightarrow 0$ is a strictly exact sequence. Hence strictly exact functors preserve strict epimorphisms. Given any $\mor_2\colon\qobj_2\rightarrow\qobj_3$, a morphism $\mor_1\colon\qobj_1\rightarrow\qobj_2$ is a kernel of $\mor_2$, iff in $0\rightarrow\qobj_1\overset{\mor_1}\longrightarrow\qobj_2\overset{\mor_2}\longrightarrow\qobj_3$ is strictly exact. \hide{Indeed, exactness of the left pair implies that $\kernel{\mor_1}=0$, i.e.\@ $\mor_1$ is a monomorphism, but then $\qobj_1\rightarrow\kernel{\mor_2}$ has to be a monomorphism as well, now using exactness of the right pair we see that $\qobj_1\rightarrow\kernel{\mor_2}$ is both a strict monomorphism and a strict epimorphism, i.e.\@ it is an isomorphism. }%
Thus strictly exact functors preserve kernels, and being additive they preserve all finite limits. 

\begin{proposition}\label{ConditionsStrictlyExact} An additive functor $\adifu\colon\quabi\rightarrow\quabi'$ between quasi-abelian categories is strictly exact, iff it preserves finite limits and strict epimorphisms.\end{proposition}
\begin{proof} Suppose $\adifu$ preserves finite limits and strict epimorphisms, let $\qobj_1\overset{\mor_1}\longrightarrow\qobj_2\overset{\mor_2}\longrightarrow\qobj_3$ in $\quabi$ be strictly exact. Then $\mor_1$ factors into $\qobj_1\rightarrow\kernel{\mor_2}\rightarrow\qobj_2$, and the first arrow is a strict epimorphism. By assumption $\adifu(\kernel{\mor_2})\cong\kernel{\adifu(\mor_2)}$, and $\adifu(\qobj_1)\rightarrow\kernel{\adifu(\mor_2)}$ is a strict epimorphism.\end{proof}

Finite colimits do not have to be preserved by strictly exact functors. On the other hand, if $\adifu$ preserves both finite limits and finite colimits, it has to be strictly exact. Indeed, $\qobj_1\overset{\mor_1}\longrightarrow\qobj_2\overset{\mor_2}\longrightarrow\qobj_3$ is strictly exact, if $\qobj_1\rightarrow\kernel{\mor_2}$ is a cokernel. Then $\adifu(\qobj_1)\rightarrow\adifu(\kernel{\mor_2})$ is a cokernel as well, i.e.\@ a strict epimorphism, and $\adifu(\kernel{\mor_2})\cong\kernel{\adifu(\mor_2)}$.

\smallskip

In performing homological computations in quasi-abelian categories we will use the following obvious facts.

\begin{proposition}\label{ManipulationWithKernels} Let $\quabi$ be a quasi-abelian category.\begin{enumerate}
\item Let $\qobj_1\rightarrow\qobj$, $\qobj_2\rightarrow\qobj$ be strict monomorphisms in $\quabi$, then $\qobj_1\underset{\qobj}\times\qobj_2\rightarrow\qobj$ is a strict monomorphism, which we denote by $\qobj_1\cap\qobj_2$.
\item Let $\{\phi_i\colon\underset{0\leq j\leq n}\bigoplus\qobj_j\rightarrow\qobj'_i\}_{1\leq i\leq n}$ be s.t.\@ $\phi_i|_{\qobj_{j\neq i}}=0$ and $\phi_i|_{\qobj_i}$ is a monomorphism. Then  $\underset{1\leq i\leq n}\bigcap\kernel{\phi_i}\cong\qobj_0$.
\end{enumerate}\end{proposition}
\begin{proof} 1. $\qobj_1\underset{\qobj}\times\qobj_2$ is the kernel of $\qobj\rightarrow(\qobj/\qobj_1)\oplus(\qobj/\qobj_2)$. 

\noindent 2. $\underset{1\leq i\leq n}\bigcap\kernel{\phi_i}\rightarrow\underset{0\leq j\leq n}\bigoplus\qobj_j$ is a sum of $\{\iota_j\colon\kernel{\phi}\rightarrow\qobj_j\}_{0\leq j\leq n}$. The composite $\underset{1\leq i\leq n}\bigcap\kernel{\phi_i}\rightarrow\underset{0\leq j\leq n}\bigoplus\qobj_j\rightarrow\underset{1\leq i\leq n}\bigoplus\qobj'_i$ is $\underset{1\leq i\leq n}\sum\phi_i\circ\iota_i=0$. Thus $\underset{1\leq i\leq n}\bigcap\kernel{\phi_i}$ factors through $\qobj_0$. Obviously the inclusion of $\qobj_0$ factors through $\underset{1\leq i\leq n}\bigcap\kernel{\phi_i}$. Both inclusions are strict monomorphisms.\end{proof}

Often it is useful to pass from quasi-abelian categories to abelian ones. There are two universal ways of doing this, and since we will be using simplicial homotopy theory (i.e.\@ left resolutions) we choose the left one.

\begin{definition}\label{LeftAbelianEnvelope} (\cite{Schneiders99} Prop.\@ 1.2.35) Let $\quabi$ be a quasi-abelian category. {\it A left abelian envelope} of $\quabi$ is given by an abelian category $\abi$ and a functor $\abiono\colon\quabi\rightarrow\abi$, s.t.\@ $\abiono$ is full and faithful; $\forall Q\in\quabi$ and for any monomorphism $A\rightarrow\abion{Q}$ there is $Q'\in\quabi$ $A\cong\abion{Q'}$; $\forall A\in\abi$ there is $Q\in\quabi$ and an epimorphism $\abion{Q}\rightarrow A$.\end{definition}
As it is shown in loc.\@ cit.\@ any quasi-abelian category has a left abelian envelope, which is unique up to an equivalence. The left abelian envelope of $\coconbo$ was explicitly constructed by L.Waelbroeck (e.g.\@ \cite{WN05}).

\begin{proposition}\label{GoodEnvelopes} (\cite{Schneiders99} Prop.\@ 1.2.26, Cor.\@ 1.2.27) Let $\quabi$ be a quasi-abelian category, and let $\abiono\colon\quabi\rightarrow\abi$ be a left abelian envelope. Then $\abiono\colon\quabi\rightarrow\abi$ has a left adjoint, and $Q_1\rightarrow Q_2\rightarrow Q_3$ in $\quabi$ is strictly exact, if and only if $\abion{Q_1}\rightarrow\abion{Q_2}\rightarrow\abion{Q_3}$ is exact in $\abi$.\end{proposition}
As $\abiono$ has a left adjoint it commutes with all limits, since it is also full and faithful it can be used to translate the classical Dold--Kan correspondence into the quasi-abelian setting.

\begin{proposition}\label{DoldKan} (\cite{Kelly16}, Cor.\@ 5.18) Let $\quabi$ be a quasi-abelian category, and let $\simpo{\quabi}$, $\compo{\quabi}$ be the categories of simplicial objects and non-positively graded complexes in $\quabi$ respectively. For any $\simo{Q}\in\simpo{\quabi}$ define $\normach{\simo{Q}}^{-k}:=\underset{0\leq i< k}\bigcap\kernel{\boundaryo_i}$ with $\boundaryo_k$'s providing the differentials. Conversely, for any $\como{Q}\in\compo{\quabi}$ define $\simcho{\como{Q}}_k:=\underset{\finor{k}\twoheadrightarrow\finor{m}}\bigoplus Q^{-m}$. Then $(\normacho,\simchoo)$ is an equivalence of categories.\end{proposition}

\subsection{Projective objects and monoidal structure}\label{ProjectiveObjectsMonoidalStructure}

\begin{definition} (E.g.\@ \cite{Schneiders99} Def.\@ 1.3.18) An object $\obj$ in an additive category $\adica$ is {\it projective}, if $\yoneda{\obj}{\adica}$ maps strict epimorphisms to surjections in the category $\abel$ of abelian groups. An additive category {\it has enough projectives}, if every object is a quotient of a projective one.\end{definition}
In any category $\adica$ and for any object $\obj$ the functor $\yoneda{\obj}{\adica}$ preserves all limits. Therefore in a quasi-abelian category $\quabi$ an object $\qobj$ is projective, iff $\yoneda{\qobj}{\quabi}\colon\quabi\rightarrow\abel$ is strictly exact (Prop.\@ \ref{ConditionsStrictlyExact}).

\smallskip

Now we look at the projective objects in $\banco$ and $\coconbo$. 

\begin{proposition} (\cite{Koethe66}, \S 3(7)) A Banach space is projective in $\banco$, if and only if it is isomorphic to $\summable{\set}$ for some set $\set$. Here $\summable{\set}$ is the space of absolutely summable $\set$-families of real numbers, i.e.\@ maps $a\colon\set\rightarrow\mathbb R$, s.t.\@ $\cardi{\{\element\in\set\,|\,a(\element)\neq 0\}}\leq\aleph_0$ and $\underset{\element\in\set}\sum |a(\element)|<\infty$.\end{proposition}

To analyze projective objects in $\coconbo$ we need to decompose bornological spaces into Banach pieces. This is called {\it dissection} (e.g.\@ \cite{Meyer2007} \S 1.5). Explicitly, every $\bornospace\in\coconbo$ equals a filtered union of Banach spaces $\{\linearspan{\disk}{\bornospace}\}_{\disk\in\boudis{\bornospace}}$ where $\boudis{\bornospace}$ consists of bounded disks $\disk\subseteq\bornospace$, s.t.\@ $(\linearspan{\disk}{\bornospace},\gaugenorm{\disk})$ is Banach.

\begin{proposition} (E.g.\@ \cite{Meyer2007}, Thm.\@ 1.139 (1)-(7)) The dissection functor defines an equivalence between $\coconbo$ and a full reflective subcategory of the category $\indo{\banco}$ of Ind-Banach spaces. The left adjoint functor is given by computing the colimits in $\coconbo$.\end{proposition}

Categories of Ind-objects are equivalent to categories of Ind-representable pre-sheaves (\cite{SGA4.1}, \S I.8.2). Therefore, since colimits of pre-sheaves are computed object-wise, $\forall\banachspace\in\banco$, $\forall\bornospace\in\coconbo$
	\begin{equation}\label{homdiss}\underset{\disk\in\boudis{\bornospace}}\colim\hom_{\banco}(\banachspace,\linearspan{\disk}{\bornospace})
	\overset{\cong}{\longrightarrow}\hom_{\coconbo}(\banachspace,\bornospace),\end{equation}
A morphism $\bornomor\colon\bornospace\rightarrow\bornospace'$ in $\coconbo$ is a strict epimorphism, if and only if it is surjective and $\{\bornomor(\disk)\}_{\disk\in\boudis{\bornospace}}$ is cofinal in $\boudis{\bornospace'}$. In this case
	\begin{equation}\label{homstrong}
	\underset{\disk\in\boudis{\bornospace}}\colim\hom_{\banco}(\banachspace,\linearspan{\bornomor(\disk)}{\bornospace'})
	\cong\underset{\disk'\in\boudis{\bornospace'}}\colim\hom_{\banco}(\banachspace,\linearspan{\disk'}{\bornospace'}).\end{equation}
The functor $\colim\colon\abel^{\boudis{\bornospace}}\rightarrow\abel$ has a right adjoint, therefore, if $\banachspace$ is projective in $\banco$, $\underset{\disk\in\boudis{\bornospace}}\colim\hom_{\banco}(\banachspace,\linearspan{\disk}{\bornospace})\rightarrow\underset{\disk\in\boudis{\bornospace}}\colim\hom_{\banco}(\banachspace,\linearspan{\bornomor(\disk)}{\bornospace'})$ is surjective. Together (\ref{homdiss}) and (\ref{homstrong}) imply the following result.

\begin{proposition} (E.g.\@ proof of Lemma 2.12 in \cite{ProsmansSchneiders00}) If $\banachspace\in\banco$ is projective in $\banco$, it is also projective in $\coconbo$.\end{proposition}

By definition $\hom_{\coconbo}(\underset{i\in I}\bigoplus\banachspace_i,\bornospace)\cong\underset{i\in I}\prod\hom_{\coconbo}(\banachspace_i,\bornospace)$, therefore $\underset{i\in I}\bigoplus\banachspace_i$ is projective in $\coconbo$, if each $\banachspace_i$ is projective in $\banco$.

Given a family of objects in $\coconbo$ their direct sum is the direct sum of the underlying vector spaces together with the smallest bornology containing images of the bounded disks in the summands (e.g.\@ \cite{Meyer2007} \S 1.3.5). Explicitly a base of the direct sum bornology consists of finite sums of images of elements of bases in the summands. Therefore, for any $\bornospace\in\coconbo$ the canonical morphism $\underset{\disk\in\boudis{\bornospace}}\bigoplus\linearspan{\disk}{\bornospace}\rightarrow\bornospace$ is a strict epimorphism. 

Let $\banachspace\in\banco$, recall (\cite{Hewitt43} Def.\@ 18) that the {\it density character} of $\banachspace$ is the smallest cardinal $\dechar{\banachspace}$, s.t.\@ there is a dense subset $\set\subseteq{\banachspace}$ with $\cardi{\set}=\dechar{\banachspace}$. As $\mathbb Q\subset\mathbb R$ is dense and countable, it is clear that, if $\dim\banachspace\geq\infty$, $\dechar{\banachspace}$ equals the smallest cardinality of a subset $\set$ of the unit sphere $\spherun{\banachspace}\subseteq\banachspace$, s.t.\@ $\forall\avector\in\banachspace$ $\exists\{\avector_k\}_{k\in\mathbb N}\subseteq\set$ $\avector=\underset{k\in\mathbb N}\sum a_k\avector_k$ with $\{a_k\}_{k\geq 1}\subset\mathbb R$. Having chosen such $\set$ we have a strict epimorphism $\summable{\set}\rightarrow\banachspace$ defined by the identity on $\set$. In particular $\banco$ has enough projectives (e.g.\@ \cite{Prosmans95} Prop.\@ 3.2.2). \hide{Since norms of elements of $\set$ in $\banachspace$ are $1$, any summable sequence in $\summable{\set}$ is also summable in $\banachspace$, this defines a bounded linear morphism. This morphism is surjective since every element in $\banachspace$ is a sum of elements in $\set$.
}%

Choosing such $\set$ in each $\disk\in\boudis{\bornospace}$ we obtain a family $\{\summable{\set_\disk}\rightarrow\linearspan{\disk}{\bornospace}\}_{\disk\in\boudis{\bornospace}}$ of strict epimorphisms. Since direct sums preserve cokernels, altogether we have the following fact.

\begin{proposition} (E.g.\@ \cite{ProsmansSchneiders00} Prop.\@ 5.8) Every $\bornospace\in\coconbo$ is a quotient of some $\underset{i\in I}\bigoplus\,\summable{\set_i}$. In particular $\coconbo$ has enough projective objects.\end{proposition}

Every projective object in $\coconbo$ is a retract of a direct sum of $\summable{\set}$'s. We would like to formalize this property.

\begin{definition}\label{GeneratingProjectives} Let $\adica$ be an additive category. A class of projective objects $\mathcal P\subseteq\adica$ is {\it sufficiently large}, if it is closed with respect to finite direct sums and every projective object in $\adica$ is a retract of some object in $\mathcal P$. {\it A generating projective class} in $\adica$, is a class of projectives $\mathcal P$, s.t.\@ the class of all small direct sums of objects in $\mathcal P$ is sufficiently large.\end{definition}

\medskip

Now we look at the monoidal structure and monoids in $\coconbo$.

\begin{proposition} (E.g.\@ \cite{Meyer2007} \S1.3.6) The category $\coconbo$ of complete convex bornological spaces has a closed symmetric monoidal structure, where $\bornospace_1\ctensor\bornospace_2$ represents the functor of bounded bilinear morphisms out of $\bornospace_1\oplus\bornospace_2$.\end{proposition}

The following definition is standard.

\begin{definition} {\it A commutative monoid} in a symmetric monoidal category $(\monoca,\monote,\monone)$ is a triple $(\monoid,\monopro{\monoid},\monou{\monoid})$, where $\monopro{\monoid}\colon\monoid\monote\monoid\rightarrow\monoid$, $\monou{\monoid}\colon\monone\rightarrow\monoid$ satisfy the usual axioms of associativity, commutativity and unitality (e.g.\@ \cite{Marty09} Def.\@ 1.2.8). The category of such monoids will be denoted by $\ucori{\monoca}$.\end{definition}

If $\monoca$ is {\it closed} symmetric monoidal, the forgetful functor $\forget\colon\ucori{\monoca}\rightarrow\monoca$ has a left adjoint $\freecou\colon\monoca\rightarrow\ucori{\monoca}$ (e.g.\@ \cite{Marty09} Prop.\@ 1.3.1). Any limits and colimits that might exist in $\monoca$ also exist in $\ucori{\monoca}$ (e.g.\@ \cite{Marty09} Prop.\@ 1.2.14 and \cite{Johnstone02} Lemma 1.1.8 \S C1.1 and discussion thereafter). The monoidal structure being closed also implies the following well known facts.

\begin{proposition}\label{EffectiveStrict} Let $\monoca$ be a quasi-abelian, closed symmetric monoidal category with all finite limits and colimits.\begin{enumerate}
\item For any morphism $\phi\colon\monoid'\rightarrow\monoid''$ in $\ucori{\monoca}$ we have $\phi=i\circ\pi$, s.t.\@ $\forget(i)$ is a monomorphism and  $\forget(\pi)$ is a strict epimorphism in $\monoca$.
\item A morphism $\phi$ in $\ucori{\monoca}$ is an effective epimorphism, if and only if $\forget(\phi)$ is a strict epimorphism.\end{enumerate}\end{proposition}
\begin{proof} 1. Let $K\rightarrow\forget(\monoid')$ be the kernel of $\forget(\phi)$, consider the diagram of solid arrows
	\begin{equation*}\xymatrix{\forget(\monoid')\monote\forget(\monoid')\ar[d]_{\monopro{\monoid'}}\ar[r]_{\pi\monote\pi\qquad} & 
	(\forget(\monoid')/K)\monote(\forget(\monoid')/K)\ar@{.>}[d]_{\monopro{\monoid'/K}}\ar[r]_{\qquad i\monote i} & 
	\forget(\monoid'')\monote\forget(\monoid'')\ar[d]^{\monopro{\monoid''}}\\
	\forget(\monoid')\ar[r]^\pi & \forget(\monoid')/K\ar[r]^i & \forget(\monoid'').}\end{equation*}
Since $\forget(\monoid')\monote-$ preserves cokernels, $\forget(\monoid')\monote\forget(\monoid')\rightarrow\forget(\monoid')\monote(\forget(\monoid')/K)$ is a cokernel of $\forget(\monoid')\monote K\rightarrow\forget(\monoid')\monote\forget(\monoid')$, i.e.\@ it is a strict epimorphism, similarly $\forget(\monoid')\monote(\forget(\monoid')/K)\rightarrow(\forget(\monoid')/K)\monote(\forget(\monoid')/K)$ is a cokernel of $K\monote(\forget(\monoid')/K)\rightarrow\forget(\monoid')\monote(\forget(\monoid')/K)$. Since strict epimorphisms are closed with respect to composition (\cite{Schneiders99} Prop.\@ 1.1.7) we conclude that $\pi\monote\pi$ is a strict epimorphism. 

Commutativity of the diagram of solid arrows implies that the kernel of $\pi\monote\pi$ factors through the kernel of $\phi\circ\monopro{\monoid'}$. Working in $\monoca/\forget(\monoid'')$ and using functoriality of cokernels, we obtain $\monopro{\monoid'/K}$ making the entire diagram commutative. It is easy to see that $\monopro{\monoid'/K}$ and $\monone\rightarrow\forget(\monoid')\rightarrow\forget(\monoid')/K$ make $\forget(\monoid')/K$ into a commutative monoid. \hide{ Indeed, associativity of $\monopro{\monoid'}$ implies that the two possible morphisms $\forget(\monoid')\monote\forget(\monoid')\monote\forget(\monoid')\rightarrow\forget(\monoid')/K$ are equal. This morphism factors through $(\forget(\monoid')/K)\monote(\forget(\monoid')/K)\monote(\forget(\monoid')/K)$ in two possible ways, which have to be equal because of the universal property of cokernels. Similarly for the unit and commutativity axioms. }%
Moreover $\pi$ and $i$ are morphisms of monoids. Using \cite{Schneiders99} Prop.\@ 1.1.4 we are done.

\noindent 2. Suppose $\phi\colon\monoid'\rightarrow\monoid''$ is an effective epimorphism, i.e.\@ it is a coequalizer of some $\monoid\rightrightarrows\monoid'$. Applying part {\bf 1.\@} we obtain $\phi=i\circ\pi$. As $\forget(i)$ is a monomorphism and $\forget$ is faithful, $\pi$ is a co-cone on $\monoid\rightrightarrows\monoid'$, thus the universal property of coequalizers implies that $i$ has a right inverse. Since $\forget(i)$ is a monomorphism it has to be invertible.

Suppose that $\forget(\phi)\colon\forget(\monoid')\rightarrow\forget(\monoid'')$ is a strict epimorphism. Let $\monoid:=\monoid'\underset{\monoid''}\times\monoid'$, and let $\overline\phi\colon\monoid'\rightarrow\overline{\monoid'}$ be the coequalizer of $\monoid\rightrightarrows\monoid'$. Since $\phi$ is a co-cone on $\monoid\rightrightarrows\monoid'$ we have $\phi=\psi\circ\overline{\phi}$. Since $\forget$ has a left adjoint we have $\forget(\monoid)\cong\forget(\monoid')\underset{\forget(\monoid'')}\times\forget(\monoid')$. Being a strict epimorphism $\forget(\phi)$ is a coequalizer of $\forget(\monoid')\underset{\forget(\monoid'')}\times\forget(\monoid')\rightrightarrows\forget(\monoid')$, therefore $\forget(\psi)$ has a left inverse. Since $\forget(\psi)$ is a strict epimorphism (\cite{Schneiders99} Prop.\@ 1.1.8) it must be invertible.\end{proof}

\subsection{Locally separable spaces}\label{LocallySeparableSpaces}

It is often important to limit the possible density characters of Banach spaces under consideration. Arguably the most important class of bornological spaces are the locally separable ones. 

\begin{definition} For a cardinal\footnote{We view a cardinal as the smallest ordinal of a given cardinality.} $\cardinal$ let $\recoco{\cardinal}\subseteq\coconbo$ be the full subcategory consisting of $\bornospace\in\recoco{\cardinal}$, s.t.\@ there is a strict epimorphism $\underset{i\in I}\bigoplus\,\summable{\cardinal_i}\rightarrow\bornospace$ with $\cardinal_i<\cardinal$ for every $i\in I$.\end{definition}
The property $\bornospace\in\recoco{\cardinal}$ can be alternatively expressed as follows: in the dissection of $\bornospace$ into a filtered union of Banach spaces, there is a cofinal family consisting of quotients of $\summable{\cardinal'}$'s with $\cardinal'<\cardinal$.

\begin{example}\label{SizeExamples}\begin{enumerate}
\item Clearly $\recoco{\aleph_0}$ consists of {\it fine bornological spaces}, whose bounded disks are bounded disks in finite dimensional subspaces.
\item By definition $\recoco{\aleph_1}$ consists of {\it locally separable} bornological spaces (e.g.\@ \cite{Meyer2007} Def.\@ 1.162).\end{enumerate}\end{example}

\begin{proposition}\label{CoreflectiveRestriction} Let $\cardinal\geq\aleph_0$, then $\recoco{\cardinal}\hookrightarrow\coconbo$ preserves finite limits and has a right adjoint. The product $\ctensor\colon\coconbo\times\coconbo\rightarrow\coconbo$ maps $(\recoco{\cardinal})^{\times^2}$ to $\recoco{\cardinal}$, making $\recoco{\cardinal}$ into a closed symmetric monoidal subcategory.\end{proposition}
\begin{proof} First we construct the right adjoint. Let $\bornospace\in\coconbo$ and define $\resize{\bornospace}\cardinal$ to have the same underlying $\mathbb R$-vector space as $\bornospace$, with the bornology given by $\disk\in\boudis{\bornospace}$ s.t.\@ there is a strict epimorphism $\summable{\cardinal'}\rightarrow\linearspan{\disk}{\bornospace}$ with $\cardinal'<\cardinal$. We claim this family of disks defines a bornology.

Indeed, it is clear that if $\disk$ is bounded in $\resize{\bornospace}{\cardinal}$, $\forall r\in\mathbb R_{>0}$ $r\cdot\disk$ is also bounded in $\resize{\bornospace}{\cardinal}$. Since $\summable{\cardinal'\sqcup\cardinal''}\cong\summable{\cardinal'}\oplus\summable{\cardinal''}$ and $\cardinal$ is infinite, it is clear that if $\disk'$, $\disk''$ are bounded in $\resize{\bornospace}{\cardinal}$, so is $\disk'+\disk''$. Finally if $\disk\subseteq\bornospace$ is a bounded disk in a finite dimensional subspace of $\bornospace$, it is bounded in $\resize{\bornospace}{\cardinal}$, thus the new family of bounded disks exhausts $\bornospace$. 

\smallskip

Let $\bornomor\colon\bornospace\rightarrow\bornospace'$ be a morphism in $\coconbo$, and let $\disk\in\boudis{\bornospace}$ be s.t.\@ we have a strict epimorphism $\summable{\cardinal'}\rightarrow\linearspan{\disk}{\bornospace}$ with $\cardinal'<\cardinal$. We claim $\exists\disk'\in\boudis{\bornospace'}$ s.t.\@ $\bornomor(\disk)\subseteq\disk'$ and $\disk'$ is bounded in $\resize{\bornospace'}{\cardinal}$. Indeed, as $\bornomor$ is bounded $\exists\disk''\in\boudis{\bornospace'}$ s.t.\@ $\bornomor(\disk)\subseteq\disk''$. If $\bornomor(\linearspan{\disk}{\bornospace})$ is closed in $\linearspan{\disk''}{\bornospace'}$, it is a Banach subspace, and the composite $\summable{\cardinal'}\rightarrow\bornomor(\linearspan{\disk}{\bornospace})$ is a strict epimorphism. Taking $\disk'$ to be the unit ball in $\bornomor(\linearspan{\disk}{\bornospace})$ we are done. 

If $\bornomor(\linearspan{\disk}{\bornospace})$ is not closed in $\linearspan{\disk''}{\bornospace'}$, $\cardinal'$ has to be infinite and we can find a dense $\set\subseteq\linearspan{\disk}{\bornospace}$ s.t.\@ $\cardi{\set}\leq\cardinal'$. Let $\disk'$ be the unit ball in the closure of the linear span of the image of $\set$ in $\linearspan{\disk''}{\bornospace'}$. Clearly $\disk'$ has a dense subset of cardinality $\leq\cardinal'$, and $\bornomor(\disk)\subseteq\disk'$.

Therefore the construction $\bornospace\mapsto\resize{\bornospace}\cardinal$ is a functor, right adjoint to the inclusion $\recoco{\cardinal}\subseteq\coconbo$. Indeed, any bounded linear map $\bornospace\rightarrow\bornospace'$ in $\coconbo$ is also bounded as a map $\resize{\bornospace}\cardinal\rightarrow\resize{\bornospace'}\cardinal$, and since $\resize{\bornospace}\cardinal=\bornospace$, if $\bornospace\in\recoco{\cardinal}$, we have $\hom_{\coconbo}(\bornospace,\bornospace')=\hom_{\recoco{\cardinal}}(\bornospace,\resize{\bornospace'}{\cardinal})$.

\smallskip

Being a full coreflective subcategory of $\coconbo$, $\recoco{\cardinal}$ is closed with respect to all colimits in $\coconbo$, and moreover they coincide with the colimits computed within $\recoco{\cardinal}$ itself. It also has all limits, which can be computed by taking them within $\coconbo$ and then applying the right adjoint.

As finite sums equal finite products in $\coconbo$, $\recoco{\cardinal}\hookrightarrow\coconbo$ preserves finite direct products. It also preserves kernels. To show this we need to compare strict monomorphisms in $\coconbo$ and $\recoco{\cardinal}$. Let $\bornospace'\subseteq\bornospace$ be a strict monomorphism in $\coconbo$ with $\bornospace\in\recoco{\cardinal}$. The bornology on $\bornospace'$ is generated by $\{\bornospace'\cap\disk\}_{\disk\in\boudis{\bornospace}}$. Let $\disk\in\boudis{\bornospace}$ s.t.\@ we have a strict epimorphism $\summable{\cardinal'}\rightarrow\linearspan{\disk}{\bornospace}$ with $\cardinal'<\cardinal$. If $\cardinal'<\aleph_0$, clearly $\linearspan{\bornospace'\cap\disk}{\bornospace'}$ is finite dimensional. If $\cardinal'\geq\aleph_0$, then $\dechar{\linearspan{\disk\cap\bornospace'}{\bornospace'}}\leq\cardinal'$ since in metric spaces density equals weight (i.e.\@ the least cardinality of a base of the topology). Therefore $\bornospace'=\resize{\bornospace'}{\cardinal}$.

\smallskip

In $\banco$ we have $\summable{\set}\ctensor\summable{\set'}\cong\summable{\set\times\set'}$ (e.g.\@ \cite{Prosmans95} Prop.\@ 3.3.5), on the other hand each Banach space constitutes an inductive system in $\banco$, therefore, using e.g.\@ \cite{Meyer2007} Thm.\@ 1.139(6-7), we conclude that the inclusion $\banco\subset\coconbo$ is symmetric monoidal, hence $\summable{\set}\ctensor\summable{\set'}\cong\summable{\set\times\set'}$ also in $\coconbo$. Since $\ctensor$ is a closed symmetric monoidal structure, it preserves colimits, and in particular strict epimorphisms. Therefore $\recoco{\cardinal}$ is closed in $\coconbo$ with respect to $\ctensor$. It is immediate to see that $\resize{\Hom(\bornospace,\bornospace')}{\cardinal}$ completes $\ctensor$ to a closed symmetric monoidal structure on $\recoco{\cardinal}$.\end{proof}%

\smallskip

The following lemma is obvious.

\begin{lemma}\label{CardinalRetract} Let $\cardinal'\leq\cardinal$ be two cardinals. Then $\summable{\cardinal'}$ is a retract of $\summable{\cardinal}$. A morphism in $\coconbo$ has the right lifting property with respect to $\summable{\cardinal'}$ for every $\cardinal'\leq\cardinal$, iff it has the right lifting property with respect to $\summable{\cardinal}$.\end{lemma}
\hide{\begin{proof} As $\cardinal\geq\cardinal'$ we can decompose $\cardinal=\cardinal'\sqcup\cardinal''$ with $\cardinal''\leq\cardinal$. Every absolutely summable $a'\colon\cardinal'\rightarrow\mathbb R$ extends by $0$ on $\cardinal''$ to an absolutely summable $a\colon\cardinal\rightarrow\mathbb R$. This extension is clearly linear and norm preserving, i.e.\@ it is a morphism of Banach spaces. 

Given an absolutely summable $a\colon\cardinal\rightarrow\mathbb R$ and restricting to $\cardinal'$ we obtain a linear map $\summable{\cardinal}\rightarrow\summable{\cardinal'}$ since restrictions of functions are linear and sub-series of absolutely convergent series are themselves absolutely convergent. Clearly the norm of a restriction is not greater than the norm of the whole series, i.e.\@ this is a morphism of Banach spaces.

Obviously the extension morphism is a section of the restriction morphism.

The only if direction is obvious. The if direction follows from the standard argument for retracts.\end{proof}}%

The following statement follows immediately from Prop.\@ \ref{CoreflectiveRestriction}.

\begin{proposition}\label{RestrictedProjectiveClass} Let $\cardinal$ be an infinite cardinal. A morphism in $\recoco{\cardinal}$ is a cokernel, if and only if it has the right lifting property with respect to $\summable{\cardinal'}$ for each $\cardinal'<\cardinal$. An object in $\recoco{\cardinal}$ has the l.l.p.\@ with respect to all cokernels in $\recoco{\cardinal}$, if and only if it is projective, i.e.\@ it is a retract of $\underset{i\in I}\bigoplus\,\summable{\cardinal_i}$ with $\cardinal_i<\cardinal$ for each $i\in I$.\end{proposition}
\begin{proof} Cokernels in $\recoco{\cardinal}$ can be equivalently computed in $\coconbo$, thus a morphism in $\recoco{\cardinal}$ is a cokernel, if and only if it is a strict epimorphism in $\coconbo$. Strict epimorphisms in $\coconbo$ are exactly the ones that have the right lifting property with respect to all projective objects in $\coconbo$. Since $\summable{\cardinal}$ is projective, the only if direction is clear.

Suppose $\bornomor\colon\bornospace\rightarrow\bornospace'$ in $\recoco{\cardinal}$ has the right lifting property with respect to $\summable{\cardinal'}$ $\forall\cardinal'<\cardinal$. Let $I$ be any set, and let $\bornomor'\colon\underset{i\in I}\bigoplus\,\summable{\cardinal_i}\rightarrow\bornospace'$ be any morphism in $\recoco{\cardinal}$ with $\forall i\,\cardinal_i<\cardinal$. Since it can be equivalently described as a family $\{\bornomor'_i\colon\summable{\cardinal_i}\rightarrow\bornospace'\}_{i\in I}$, it is clear that $\bornomor$ factors $\bornomor'$, i.e.\@ $\bornomor$ has the right lifting property with respect to $\underset{i\in I}\bigoplus\,\summable{\cardinal_i}$.

Let $\bornomor''\colon\bornospace''\rightarrow\bornospace'$ be any morphism in $\coconbo$ with $\bornospace''$ being projective. Since $\bornospace'\in\recoco{\cardinal}$ there is a strict epimorphism $\bornomor'\colon\underset{i\in I}\bigoplus\,\summable{\cardinal_i}\rightarrow\bornospace'$ with $\forall i\,\cardinal_i<\cardinal$. Since $\bornospace''$ is projective $\bornomor'$ factors $\bornomor''$, hence $\bornomor$ factors $\bornomor''$, i.e.\@ $\bornomor$ has the right lifting property with respect to all projective objects in $\coconbo$.

Let $\bornospace\in\recoco{\cardinal}$ have the l.l.p.\@ with respect to all cokernels. There is $I$ and a strict epimorphism $\underset{i\in I}\bigoplus\,\summable{\cardinal_i}\rightarrow\bornospace$ with $\forall i\,\cardinal_i<\cardinal$. By assumption this epimorphism has a section, realizing $\bornospace$ as a retract of $\underset{i\in I}\bigoplus\,\summable{\cardinal_i}$.\end{proof}

\begin{example} In $\recoco{\aleph_0}$ every object is projective. Indeed, every $\mathbb R$-vector space is an infinite sum of copies of $\mathbb R$.\end{example}

\begin{proposition}\label{OneProjective} Let $\cardinal\geq\aleph_0$. The category $\recoco{\cardinal}$ is quasi-abelian, complete and cocomplete with enough projectives. The set $\{\summable{\cardinal'}\}_{\cardinal'<\cardinal}$ is a generating projective class in $\recoco{\cardinal}$ (Def.\@ \ref{GeneratingProjectives}). The completed projective tensor product $\ctensor$ makes $\recoco{\cardinal}$ into a closed symmetric monoidal category.\end{proposition}
\begin{proof} From Prop.\@ \ref{CoreflectiveRestriction} we know that $\recoco{\cardinal}$ is complete and cocomplete (because $\coconbo$ is such). Also $\recoco{\cardinal}\hookrightarrow\coconbo$ preserves finite limits and colimits, hence pullbacks of strict epimorphisms are strict epimorphisms in $\recoco{\cardinal}$, similarly for strict monomorphisms and pushouts. So $\recoco{\cardinal}$ is quasi-abelian. The rest follows from Prop.\@ \ref{RestrictedProjectiveClass} and the fact that objects of $\recoco{\cardinal}$ are by definition quotients of direct sums of $\summable{\cardinal_i}$'s for $\cardinal_i<\cardinal$.\end{proof}

\begin{example} We are mostly interested in $\recoco{\aleph_1}$. The singleton $\{\summable{\aleph_0}\}$ is a generating projective class in $\recoco{\aleph_1}$. We will denote this Banach space by $\samcou$ and instead of $\recoco{\aleph_1}$ we will write $\coubo$.\end{example}

\subsection{Local presentability and small objects}\label{LocalPresentabilitySmallObjects}

Estimating the size of objects is important in defining closed model structures. In our case this will appear when we will be transferring a model structure to the category of modules over monoids. In a later work we will be also concerned with objects of finite presentation.

\smallskip

For an ordinal $\ordina$ we denote by $\cardint{\ordina}$ the category of ordinals $<\ordina$ (a non-identity morphism $\ordina_i\rightarrow\ordina_j$ is $\ordina_i\in\ordina_j$). Let $\caty$ be a category, {\it an $\ordina$-sequence} in $\caty$ is any functor $\cardint{\ordina}\rightarrow\caty$ that preserves colimits. Recall that a subcategory $\caty'\subseteq\caty$ is {\it closed under transfinite compositions}, if $\forall\ordina$ and for any $\ordina$-sequence $\siqu\colon\cardint{\ordina}\rightarrow\caty'$, s.t.\@ $\underset{\cardint{\ordina}}\colim\,\siqu$ exists in $\caty$, the morphism $\siqu(\emptyset)\rightarrow\underset{\cardint{\ordina}}\colim\,\siqu$ is in $\caty'$.

\begin{example} The subcategory $\monobor$ of monomorphisms in $\coconbo$ is closed under transfinite composition. To see this let $\ordina$ be a limit ordinal and let $\siqu\colon\cardint{\ordina}\rightarrow\coconbo$ be a diagram consisting of monomorphisms, i.e.\@ injective morphisms of bornological spaces. As a vector space $\underset{\cardint{\ordina}}\colim\,\siqu$ is $\underset{i\in\cardint{\ordina}}\bigcup\bornospace_i$, a disk in this union is bounded, iff it is a bounded disk in one of the $\bornospace_i$'s. \hide{
Indeed, since sums and dilations of disks do not take us out of a vector subspace, it is clear that this is a bornology. Every bounded disk whose gauge norm in some $\bornospace_i$ is complete, still generates the same Banach subspace in the union. Therefore this is a complete bornological space. The fact that it is the colimit is obvious. }%

A similar argument shows that for any diagram $\siqu\colon\poset\rightarrow\monobor$, where $\poset$ is a directed poset,\footnote{A poset is {\it directed}, if every finite subset has an upper bound (e.g.\@ \cite{AdamekRosicky94} \S 1.A).} and any $i\in\poset$ the morphism $\siqu(i)\rightarrow\underset{\poset}\colim\,\siqu$ is in $\monobor$ (the colimit is computed in $\coconbo$). \hide{%
Indeed, computing the colimit in the category of all $\mathbb R$-vector spaces we get $\underset{i\in\poset}\bigoplus\siqu(i)$ divided by an equivalence relation, where two vectors are equivalent, if their images are equal in some $\siqu(i)$. Clearly no two vectors from the same $\siqu(i)$ are equivalent in this way, i.e.\@ every $\siqu(i)$ is mapped injectively into the colimit computed in $\mathbb R$-vector spaces. The colimit in $\coconbo$ is obtained from the colimit in $\mathbb R$-vector spaces by declaring a disk to be bounded, if it is a bounded disk in the image of some $\siqu(i)$. The requirement that $\poset$ is filtered ensures that this is a bornology.
}%
\end{example}

Recall (\cite{SchwedeShipley00} \S1) that a cardinal is {\it regular}, if it equals its own cofinality, for example for any $\cardinal\geq\aleph_0$ the successor $\cardinal^+$ is regular (e.g.\@ \cite{KomTo06} \S 10). Let $\cardinal$ be a regular cardinal and let $\caty'\subseteq\caty$ be a subcategory closed under transfinite compositions, {\it $\obj\in\caty$ is $\cardinal$-small relative to $\caty'$} (e.g.\@ \cite{SchwedeShipley00} \S1), if for any regular $\cardinal'\geq\cardinal$ and any $\cardinal'$-sequence $\{\ldots\obj_i\overset{d_i}\longrightarrow\obj_{i+1}\ldots\}$ in $\caty'$ the natural morphism
	\begin{equation}\label{SmallnessDiagram}\underset{\cardint{\cardinal'}}\colim\hom_{\caty}(\obj,\obj_i)\longrightarrow
	\hom_{\caty}(\obj,\underset{\cardint{\cardinal'}}\colim\,\obj_i)\end{equation}
is a bijection. An object is {\it small relative to $\caty'$}, if it is $\cardinal$-small for some regular cardinal $\cardinal$. An object in $\caty$ is {\it small}, if it is small relative to $\caty$. 

\begin{remark} Instead of the sequences $\cardint{\cardinal'}$ for regular $\cardinal'\geq\cardinal$ we can consider {\it $\cardinal$-directed categories}, i.e.\@ posets where every subset of cardinality $<\cardinal$ has an upper bound. For a regular $\cardinal$ an object $\obj\in\caty$ is {\it $\cardinal$-presentable} (e.g.\@ \cite{AdamekRosicky94} Def.\@ 1.13), if $\yoneda{\obj}{\caty}$ preserves colimits over $\cardinal$-directed categories. Considering only diagrams that factor through $\caty'\subseteq\caty$ we obtain the notion of {\it $\cardinal$-presentable objects relative to a subcategory}. For a regular $\cardinal'\geq\cardinal$, the sequence $\cardint{\cardinal'}$ is $\cardinal$-directed. Therefore $\cardinal$-presentable objects are $\cardinal$-small.\end{remark}

Instead of $\aleph_0$-presentable one usually says {\it finitely presentable} (e.g.\@ \cite{AdamekRosicky94} Def.\@ 1.1). We will call $\aleph_0$-small objects {\it compact}. Let $\ordina$ be a limit ordinal, and let $\cardinal$ be the cofinality of $\ordina$, i.e.\@ the smallest cardinality of a cofinal subcategory in $\cardint{\ordina}$. Clearly $\cardinal$ is regular and infinite. Therefore an object $\obj\in\caty$ is compact, iff $\yoneda{\obj}{\caty}$ preserves all sequential colimits. Thus the notions of finitely presentable and compact objects coincide (e.g.\@ \cite{AdamekRosicky94} Cor.\@ 1.7).

\begin{example}\label{SmallMonomorphisms} Every Banach space $\banachspace$ is compact relative to $\monobor$. Indeed, let $\ordina$ be a limit ordinal and let $\siqu\colon\cardint{\ordina}\rightarrow\coconbo$ be a diagram in $\monobor$. A morphism $\banachspace\rightarrow\underset{i\in\cardint{\ordina}}\bigcup\bornospace_i$ maps the unit ball in $\banachspace$ to a bounded disk, which has to be in the image of one of $\bornospace_i$'s. 

If $\bornospace\in\coconbo$ is not Banach, it cannot be compact relative to $\monobor$. Indeed, let $\siqu$ be the diagram of all Banach subspaces of $\bornospace$ given by the dissection. The identity on $\bornospace$ is represented by some $\bornospace\rightarrow\siqu(i)$, iff there is a maximal element in the diagram, i.e.\@ $\bornospace$ is Banach.\end{example}

Not even $\mathbb R$ is compact relative to all of $\coconbo$: 

\begin{example}\label{NotCompact} Let $\samcou\rightarrow\samcou\rightarrow\ldots$ be the $\aleph_0$-sequence corresponding to $\mathbb N\supset\mathbb N_{\geq 2}\supset\mathbb N_{\geq 3}\supset\ldots$. The colimit of this sequence is $0$, but $\{2^{-k}\}_{k\in\mathbb N}$ in the first copy of $\samcou$ is not identified with $0$ at any step of the sequence.\end{example}
However, $\mathbb R$ is $\aleph_1$-presentable in $\coconbo$. First we need a lemma.

\begin{lemma}\label{FullKernels} Let $\poset$ be an $\aleph_1$-directed poset, and consider a functor $\siqu\colon\poset\rightarrow\coconbo$. Given $i\rightarrow j$ in $\poset$ let $\K_{i,j}\rightarrow\siqu(i)$ be the kernel of $\siqu(i)\rightarrow\siqu(j)$. The subset $\K_i:=\underset{j}\bigcup\,\K_{i,j}$ is a closed vector subspace of $\siqu(i)$.\end{lemma} 
\begin{proof}\hide{ %
Let $a_j,a_{j'}\in\K_i$, i.e.\@ there are $i\rightarrow j$, $i\rightarrow j'$ in $\poset$ and $a_j\in\K_{i,j}$, $a_{j'}\in\K_{i,j'}$. As $\poset$ is $\aleph_1$-directed there is $i\rightarrow j''$, that factors through $j$ and $j'$. Clearly $a_j,a_{j'}\in\K_{i,j''}$ and hence $a_j+a_{j'}\in\K_i$.}%
Let $\{b_k\}_{k\in\mathbb N}\subseteq\K_i$ be a sequence that converges to $b\in\siqu(i)$. We have $b_k\in\K_{i,j_k}$ for some $\{i\rightarrow j_k\}_{k\in\mathbb N}$. As $\poset$ is $\aleph_1$-directed there is $i\rightarrow j$ that factors through each $j_k$. Then $\{b_k\}_{k\in\mathbb N}\subseteq\K_{i,j}$ and hence $b\in\K_{i,j}\subseteq\K_i$.\end{proof}%

Now we divide by these closed subspaces $\{\K_i\}$. 

\begin{lemma}\label{IntoMonomorphisms} Let $\poset$ be an $\aleph_1$-directed poset, and let $\siqu\colon\poset\rightarrow\coconbo$. For each $i\in\poset$ let $\K_i\rightarrow\siqu(i)$ be as in Lemma \ref{FullKernels}, and let $\overline{\bornospace}_i:=\siqu(i)/\K_i$. The projections $\siqu(i)\rightarrow\overline{\bornospace}_i$ extend to a natural transformation $\siqu\rightarrow\overline{\siqu}$. Moreover $\underset{\poset}\colim\,\siqu\cong\underset{\poset}\colim\,\overline{\siqu}$ and $\overline{\bornospace}_i\rightarrow\overline{\bornospace}_j$ is a monomorphism for each $i\rightarrow j$ in $\poset$. \end{lemma}
\begin{proof} Existence of $\siqu\rightarrow\overline{\siqu}$ is obvious. \hide{%
Indeed, let $i\rightarrow i'$ in $\poset$, and let $\K_{i,j}\rightarrow\siqu(i)$ be the kernel of $\siqu(i)\rightarrow\siqu(j)$ for some $i\rightarrow j$ in $\poset$. As $\poset$ is $\aleph_1$-directed there is a commutative diagram in $\poset$:
	\begin{equation*}\xymatrix{i\ar[r]\ar[d] & j\ar[d]\\ i'\ar[r] & j'.}\end{equation*}
Clearly $\K_{i,j}\subseteq\K_{i,j'}$ and the composite $\K_{i,j'}\rightarrow\siqu(i')$ factors through $\K_{i',j'}$. Therefore also the composite $\K_i\rightarrow\siqu(i')$ factors through $\K_{i'}$ and hence there is a unique $\overline{d}_{i,i'}$ making the following diagram commutative
	\begin{equation*}\xymatrix{\siqu(i)\ar[d]\ar[rr] && \siqu(i')\ar[d]\\
	\overline{\bornospace}_i\ar[rr]^{\overline{d}_{i,i'}} && \overline{\bornospace}_{i'}.}\end{equation*}
Uniqueness of $\overline{d}_{i,i'}$ implies that $\overline{\siqu}$ is a functor $\poset\rightarrow\coconbo$. }%
Consider the functor from the category of co-cones on $\overline{\siqu}$ to the category of co-cones on $\siqu$. As it is given by strict epimorphisms, this functor is obviously full and faithful. Conversely, having any co-cone $\{\phi_i\}_{i\in\poset}$ on $\siqu$, we see that $\phi_i$ maps $\K_{i,j}\rightarrow\siqu(i)$ to $0$ for every $i\rightarrow j$, therefore it factors through $\overline{\bornospace}_i$. So the functor is essentially surjective, and hence the initial co-cones for $\siqu$ and $\overline{\siqu}$ are isomorphic. Let $i\rightarrow i'$ in $\poset$, by construction an element of $\siqu(i)$ is in $\K_i$, iff its image in $\siqu(i')$ is in $\K_{i'}$. This means that $\K_i\rightarrow\siqu(i)$ is the kernel of $\siqu(i)\rightarrow\overline{\bornospace}_{i'}$, and then $\siqu(i)\rightarrow\overline{\bornospace}_i\rightarrow\overline{\bornospace}_{i'}$ is a strict epimorphism, followed by a monomorphism. \end{proof}%

Now we can show that $\mathbb R$ is $\aleph_1$-presentable in $\coconbo$. We prove a more general statement for all projective Banach spaces.

\begin{proposition} Let $\cardinal\geq\aleph_1$ be a regular cardinal, and let $\cardi{\set}<\cardinal$. The Banach space $\summable{\set}$ is $\cardinal$-presentable in $\coconbo$.\end{proposition}
\begin{proof} Let $\poset$ be an $\cardinal$-directed poset, and let $\siqu\colon\poset\rightarrow\coconbo$. Using Lemma \ref{IntoMonomorphisms} we have $\overline{\siqu}\colon\poset\rightarrow\monobor\subset\coconbo$ and a natural transformation $\siqu\rightarrow\overline{\siqu}$ consisting of strict epimorphisms. Lemma \ref{IntoMonomorphisms} tells us that $\underset{\poset}\colim\,\siqu\cong\underset{\poset}\colim\,\overline{\siqu}$ and from Example \ref{SmallMonomorphisms} we know that $\underset{i\in\poset}\colim\,\hom_{\coconbo}(\summable{\set},\overline{\siqu}(i))\cong\hom_{\coconbo}(\summable{\set},\underset{\poset}\colim\,\overline{\siqu})$. Therefore it is enough to show that 
	\begin{equation}\label{QuotientColimit}\underset{i\in\poset}\colim\,\hom_{\coconbo}(\summable{\set},\siqu(i))\longrightarrow
	\underset{i\in\poset}\colim\,\hom_{\coconbo}(\summable{\set},\overline{\siqu}(i))\end{equation}
is a bijection. As each $\siqu(i)\rightarrow\overline{\siqu}(i)$ is a strict epimorphism and $\summable{\set}$ is projective, it is clear that (\ref{QuotientColimit}) is surjective. Let $\phi_i\colon\summable{\set}\rightarrow\siqu(i)$ represent $\phi\in\underset{i\in\poset}\colim\,\hom_{\coconbo}(\summable{\set},\siqu(i))$ that is mapped to $0\in\underset{i\in\poset}\colim\,\hom_{\coconbo}(\summable{\set},\overline{\siqu}(i))$. As $\overline{\siqu}$ factors through $\monobor$, the composite $\summable{\set}\overset{\phi_i}\longrightarrow\siqu(i)\rightarrow\overline{\siqu}(i)$ must be $0$, i.e.\@ $\phi_i$ factors through $\K_i$. Since $\K_i=\underset{j}\bigcup\,\K_{i,j}$, for each $k\in\set$ there is $i\rightarrow j_k$, s.t.\@ $\phi_i(k)\in\K_{i,j_k}$. As $\poset$ is $\cardinal$-directed there is $i\rightarrow j$ that factors through each $i\rightarrow j_k$. Since $\K_{i,j}\rightarrow\siqu(i)$ is a strict monomorphism, we see that $\phi_i(\summable{\set})\subseteq\K_{i,j}$, i.e.\@ $\phi=0$ in $\underset{i\in\poset}\colim\,\hom_{\coconbo}(\summable{\set},\siqu(i))$.\end{proof}%

\begin{corollary}\label{AllProjectivePresentable} Let $\cardinal\geq\aleph_1$ be a regular cardinal, and let $\banachspace$ be a Banach space, s.t.\@ $\dechar{\banachspace}<\cardinal$. Then $\banachspace$ is $\cardinal$-presentable in $\coconbo$.\end{corollary}
\begin{proof} As $\dechar{\banachspace}<\cardinal$ there is a set $\set$ with $\cardi{\set}<\cardinal$, and a strict epimorphism $\summable{\set}\rightarrow\banachspace$. Let $\banachspace'\rightarrow\summable{\set}$ be the kernel of this morphism, then $\dechar{\banachspace'}<\cardinal$ as well, and there is $\set'$ with $\cardi{\set'}<\cardinal$, and a strict epimorphism $\summable{\set'}\rightarrow\banachspace'$. Altogether we have $\banachspace$ as the co-equalizer of $\summable{\set'}\rightrightarrows\summable{\set}$. 

\begin{lemma}\label{SmallProjectives} Let $\obj$ be a co-equalizer of $\obj''\rightrightarrows\obj'$ in a category $\caty$, let $\cardinal\geq\aleph_0$ be a regular cardinal s.t.\@ $\obj'$, $\obj''$ are $\cardinal$-presentable. Then $\obj$ is $\cardinal$-presentable.\end{lemma}
\begin{proof} Finite limits and filtered colimits commute in $\sets$. \hide{%
Let $\poset$ be an $\cardinal$-directed poset, and let $\siqu\colon\poset\rightarrow\caty$. We have a commutative diagram
	\begin{equation*}\xymatrix{\underset{i\in\poset}\colim\hom_{\caty}(\obj,\siqu(i))\ar[r]\ar[d] & 
	\hom_{\caty}(\obj,\underset{\poset}\colim\,\siqu)\ar[d]\\
	\underset{i\in\poset}\colim\hom_{\caty}(\obj',\siqu(i))\ar[r]^{\cong}\ar@<-.5ex>[d]\ar@<.5ex>[d] &
	\hom_{\caty}(\obj',\underset{\poset}\colim\,\siqu)\ar@<-.5ex>[d]\ar@<.5ex>[d]\\
	\underset{i\in\poset}\colim\hom_{\caty}(\obj'',\siqu(i))\ar[r]^{\cong} &
	\hom_{\caty}(\obj'',\underset{\poset}\colim\,\siqu).}\end{equation*}
Since $\obj$ is a colimit of $\obj''\rightrightarrows\obj'$, it is clear that the upper right vertical arrow is an equalizer of the right parallel pair. The same is true on the left at each stage $i\in\poset$. In $\sets$ filtered colimits commute with finite limits, hence the left upper vertical arrow is an equalizer of the left parallel pair. Now the claim follows from uniqueness of equalizers.}\end{proof}

Applying Lemma \ref{SmallProjectives} to $\summable{\set'}\rightrightarrows\summable{\set}$ we are done.\end{proof}%

We have seen (Ex.\@ \ref{SizeExamples}) that $\recoco{\aleph_0}$ is the category of $\mathbb R$-vector spaces, hence it is locally finitely presentable. According to Prop.\@ \ref{OneProjective} for $\cardinal\geq\aleph_0$ the category $\recoco{\cardinal}$ is cocomplete, and according to Prop.\@ \ref{CoreflectiveRestriction} the inclusion $\recoco{\cardinal}\subset\coconbo$ preserves colimits. Thus, since dissections of bornological spaces are filtered diagrams, applying Cor.\@ \ref{AllProjectivePresentable} we obtain the following. 

\begin{theorem}\label{LocalPresentability} Let $\cardinal\geq\aleph_0$ be regular, then $\recoco{\cardinal}$ is locally $\cardinal$-presentable.\footnote{A category $\caty$ is {\it locally $\cardinal$-presentable} (e.g.\@ \cite{AdamekRosicky94} Def.\@ 1.17), if it is cocomplete and there is a set of $\cardinal$-presentable objects $\{\pobj_i\}_{i\in I}\subseteq\caty$ s.t.\@ $\forall\obj\in\caty$ is an $\cardinal$-directed colimit of $\pobj_i$'s.}\end{theorem}
Theorem \ref{LocalPresentability} does not apply to the category $\coconbo$ of all complete bornological spaces because the sizes of cardinals are not bounded. However, $\coconbo$ is a nested union of locally presentable categories, and this has important consequences. The proof of the following lemma is straightforward. 
 
\begin{lemma} Let $\quabi$ be a co-complete and finitely complete category, and suppose $\forall\qobj\in\quabi$ there is an effective epimorphism $\underset{i\in I}\coprod\,\qobj_i\rightarrow\qobj$, where $I$ is a set and each $\qobj_i$ is small. Then every object in $\quabi$ is small.\end{lemma}
\hide{%
\begin{proof} Since every $\qobj\in\quabi$ is the co-equalizer of some $\underset{j\in\set'}\coprod\,\qobj'_j\rightrightarrows\underset{i\in\set}\coprod\,\qobj_i$, it is enough to show that each direct sum of small objects is small (Lemma \ref{SmallProjectives}).

By assumption $\forall i\in I$ there is a regular cardinal $\cardinal_i$, s.t.\@ $\qobj_i$ is $\cardinal_i$-small. Let $\cardinal$ be a regular cardinal, s.t.\@ $\cardinal\geq\cardi{\set}$ and $\cardinal\geq\cardinal_i$ for each $i\in\set$, and let $\siqu$ be an $\cardinal$-sequence in $\quabi$. We have an injective map
	\begin{equation}\label{InfiniteSum}\underset{j\in\cardint{\cardinal}}\colim\underset{i\in\set}\prod\hom_{\quabi}(\qobj_i,\siqu(j))\longrightarrow
	\underset{i\in\set}\prod\underset{j\in\cardint{\cardinal}}\colim\hom_{\quabi}(\qobj_i,\siqu(j)).\end{equation}
Indeed, let $\phi_j,\psi_j\colon\underset{i\in\set}\bigoplus\qobj_i\rightarrow\siqu(j)$ represent two elements on the left, that become equal on the right. This means for each $i$ the $i$-components of the images are equal, i.e.\@ for each $i$ there is $j_i$, s.t.\@ $\phi^i_{j_i}=\psi^i_{j_i}$. Since $\cardinal$ is big enough, there is $k\in\cardint{\cardinal}$, s.t.\@ $\phi_k=\psi_k$, i.e.\@ the two elements are equal also on the left.

Conversely, let $\{\phi^i\}_{i\in\set}$ be an element of the right side of (\ref{InfiniteSum}). Each component $\phi^i$ is represented by $\phi^i_{j_i}\colon\qobj_i\rightarrow\siqu(j_i)$ for some $j_i\in\cardint{\cardinal}$. Again using the fact that $\cardinal$ is big enough we can find a common bound $k\in\cardint{\cardinal}$, i.e.\@ $\{\phi^i\}_{i\in\set}$ is represented by $\{\phi^i_k\}_{i\in\set}$, and we see that (\ref{InfiniteSum}) is surjective.

Finally, using smallness of each $\qobj_i$ we see that the right hand side of (\ref{InfiniteSum}) equals $\hom_{\quabi}(\underset{i\in S}\coprod\qobj_i,\underset{\cardint{\cardinal}}\colim\,\siqu)$.\end{proof}}%

As $\forall\bornospace\in\coconbo$ lies in some $\recoco{\cardinal}$, it is a quotient of some $\underset{i\in I}\bigoplus\,\summable{\cardinal_i}$ with $\cardinal_i<\cardinal$ for each $i$. Thus the previous lemma implies the following.

\begin{proposition} Every object in $\coconbo$ is small. \end{proposition}

\subsection{Model structures}\label{ModelStructures}

It is straightforward to define a model structure on simplicial objects in an additive category with enough projectives (\cite{Qu67}). Using Prop.\@ \ref{EffectiveStrict} we can proceed in a similar fashion also with the category of commutative monoids.

The approach of \cite{Qu67} is based on {\it evaluating morphisms} $\obj_1\rightarrow\obj_2$ at some special $\{\pobj_i\}_{i\in I}\subseteq\caty$. This means looking at $\{\hom(\pobj_i,\obj_1)\rightarrow\hom(\pobj_i,\obj_2)\}_{i\in I}$. 

\begin{proposition}\label{ModelStructureInGeneral}\begin{enumerate}\item Let $\monoca$ be a quasi-abelian category with enough projectives. Let $\{P_i\}_{i\in I}$ be a generating projective class (Def.\@ \ref{GeneratingProjectives}). The category $\simpo{\monoca}$ of simplicial objects in $\monoca$ has a simplicial model structure, where $\simo{C}\rightarrow\simo{D}$ is a weak equivalence/fibration, if $\hom_{\monoca}(P_i,\simo{C})\rightarrow\hom_{\monoca}(P_i,\simo{D})$ is a weak equivalence/fibration of simplicial sets for every $P_i$. In particular every $\simo{C}\in\simpo{\monoca}$ is fibrant.

\item If in addition $\monoca$ is closed symmetric monoidal, the category $\simpo{\ucori{\monoca}}$ of simplicial commutative monoids in $\monoca$ has a simplicial model structure, and the adjunction $\freecou\colon\monoca\rightleftarrows\ucori{\monoca}\colon\forget$ extends to a Quillen adjunction $\simpo{\monoca}\rightleftarrows\simpo{\ucori{\monoca}}$. A morphism $\simo{\monoid}\rightarrow\simo{\monoid'}$ in $\simpo{\ucori{\monoca}}$ is a weak equivalence/fibration, if 
	\begin{equation*}\hom_{\simpo{\ucori{\monoca}}}(\freecou(P_i),\simo{\monoid})\longrightarrow
	\hom_{\simpo{\ucori{\monoca}}}(\freecou(P_i),\simo{\monoid'})\end{equation*}
is a weak equivalence/fibration of simplicial sets for each $P_i$. In particular every object in $\simpo{\ucori{\monoca}}$ is fibrant. \end{enumerate}\end{proposition}
\begin{proof} The first part is a straightforward application of \cite{Qu67} \S II.4. \hide{%
Since $\monoca$ has all finite limits and colimits effective epimorphisms and strict epimorphisms coincide. Then, having enough projective objects, $\monoca$ has sufficiently many projectives according to \cite{Qu67} \S II.4. Since $\monoca$ is an additive category every $\hom_{\monoca}(P,\simo{C})$ is a simplicial abelian group, hence fibrant. Therefore $\simpo{\monoca}$ is a simplicial model category with weak equivalences/fibrations determined by evaluating at all projectives $P$ (\cite{Qu67} Thm.\@ 4).

We claim it is enough to evaluate at $P_i$'s. Since the projective class in $\monoca$ is generated by $\{P_i\}$, every projective object is a retract of some $\bigoplus P_i$. Weak equivalences and fibrations are stable under retracts, hence it is enough to evaluate at direct sums of $P_i$'s. Any
	$$\hom_{\monoca}(\bigoplus P_i,\simo{C})\rightarrow\hom_{\monoca}(\bigoplus P_i,\simo{D})$$ 
given by $\simo{C}\rightarrow\simo{D}$ is a direct product of $\hom_{\monoca}(P_i,\simo{C})\rightarrow\hom_{\monoca}(P_i,\simo{D})$. Direct products of fibrations and weak equivalences between fibrant simplicial sets are morphisms of the same kind (\cite{Ho99} Cor.\@ 5.1.5), therefore it is enough to evaluate at $P_i$'s.}%
In the case of monoids Prop.\@ \ref{EffectiveStrict} tells us that $\forget$ maps effective epimorphisms in $\ucori{\monoca}$ to strict epimorphisms in $\monoca$. Therefore $\freecou(P)$ is projective in $\ucori{\monoca}$ for each projective $P$ in $\monoca$. As $\monoca$ has enough projectives, for any $\monoid\in\ucori{\monoca}$ there is a strict epimorphism $P\rightarrow\forget(\monoid)$. This strict epimorphism factors into $P\rightarrow\forget(\freecou(P))\rightarrow\forget(\monoid)$, thus $\forget(\freecou(P))\rightarrow\forget(\monoid)$ is a strict epimorphism (\cite{Schneiders99}, Prop.\@ 1.1.8). Prop.\@ \ref{EffectiveStrict} tells us then that $\freecou(P)\rightarrow\monoid$ is an effective epimorphism, i.e.\@ $\ucori{\monoca}$ has sufficiently many projectives (\cite{Qu67} \S II.4).

Let $\monoid$ be a projective object in $\ucori{\monoca}$. Choosing a strict epimorphism $P\rightarrow\forget(\monoid)$ we have an effective epimorphism $\freecou(P)\rightarrow\monoid$, and hence a section $\monoid\rightarrow\freecou(P)$. Thus every projective object in $\ucori{\monoca}$ is a retract of $\freecou(P)$ for some projective $P$ in $\monoca$. Therefore $\forall\simo{\monoid}\in\simpo{\ucori{\monoca}}$ and any projective $\monoid'\in\ucori{\monoca}$ $\hom_{\ucori{\monoca}}(\monoid',\simo{\monoid})$ is a retract of $\hom_{\monoca}(P,\forget(\simo{\monoid}))$ for some projective $P\in\monoca$. Since the latter simplicial sets are simplicial abelian groups, they are fibrant. Thus every object in $\simpo{\ucori{\monoca}}$ is fibrant, and hence $\simpo{\ucori{\monoca}}$ has a simplicial model structure (\cite{Qu67} Thm.\@ 4).

Fibrations and weak equivalences are determined by evaluating at all projective objects in $\ucori{\monoca}$. We claim it is enough to evaluate at $\freecou(P_i)$ for each $P_i$. Every projective object in $\ucori{\monoca}$ is a retract of $\freecou(P)$ for some projective $P\in\monoca$. In turn every such $P$ is a retract of some $\bigoplus P_i$, therefore it is enough to evaluate at $\freecou(\bigoplus P_i)$. For every $\simo{\monoid}$ $\hom_{\ucori{\monoca}}(\freecou(\bigoplus P_i),\simo{\monoid})\cong\hom_{\monoca}(\bigoplus P_i,\forget(\simo{\monoid}))$, thus it is enough to evaluate at $\freecou(P_i)$ for each $P_i$.\end{proof}%

\smallskip

We know that $\coconbo$ is quasi-abelian, complete and cocomplete, has enough projectives and carries a closed symmetric monoidal structure. Therefore we can apply the previous proposition to get the following.

\begin{corollary} Let $\simpo{\coconbo}$ be the category of simplicial complete convex bornological spaces over $\mathbb R$. It is a simplicial model category with $\simo{\bornospace}\rightarrow\simo{\bornospace'}$ being a weak equivalence/fibration, if $\hom_{\coconbo}(\summable{\cardinal},\simo{\bornospace})\rightarrow\hom_{\coconbo}(\summable{\cardinal},\simo{\bornospace'})$ is a weak equivalence/fibration of simplicial sets for every cardinal $\cardinal$. 

The category $\simpo{\ucori{\coconbo}}$ of simplicial commutative monoids in $\coconbo$ is a simplicial model category, with $\simo{\monoid}\rightarrow\simo{\monoid'}$ being a weak equivalence/fibration, if $\hom_{\coconbo}(\summable{\cardinal},\simo{\monoid})\rightarrow\hom_{\coconbo}(\summable{\cardinal},\simo{\monoid'})$ is a weak equivalence/fibration of simplicial sets for every $\cardinal$. All objects in $\simpo{\ucori{\coconbo}}$, $\simpo{\coconbo}$ are fibrant, the adjunction $\simpo{\coconbo}\rightleftarrows\simpo{\ucori{\coconbo}}$ is a Quillen adjunction.\end{corollary}

Using Prop.\@ \ref{OneProjective} we have the similar statement in the bounded case.

\begin{corollary}\label{ModelStructureBoundedCase} Let $\cardinal\geq\aleph_0$, let $\simpo{\recoco{\cardinal}}$ be the category of simplicial objects in $\recoco{\cardinal}$. It is a simplicial model category with $\simo{\bornospace}\rightarrow\simo{\bornospace'}$ being a weak equivalence/fibration, if $\hom_{\recoco{\cardinal}}(\summable{\cardinal'},\simo{\bornospace})\rightarrow\hom_{\recoco{\cardinal}}(\summable{\cardinal'},\simo{\bornospace'})$ is a weak equivalence/fibration of simplicial sets for every cardinal $\cardinal'<\cardinal$. 

The category $\simpo{\ucori{\recoco{\cardinal}}}$ of simplicial commutative monoids in $\recoco{\cardinal}$ is a simplicial model category, with $\simo{\monoid}\rightarrow\simo{\monoid'}$ being a weak equivalence/fibration, if $\hom_{\recoco{\cardinal}}(\summable{\cardinal'},\simo{\monoid})\rightarrow\hom_{\recoco{\cardinal}}(\summable{\cardinal'},\simo{\monoid'})$ is a weak equivalence/fibration of simplicial sets for every $\cardinal'<\cardinal$. 

All objects in $\simpo{\ucori{\recoco{\cardinal}}}$, $\simpo{\recoco{\cardinal}}$ are fibrant, the adjunction $\simpo{\recoco{\cardinal}}\leftrightarrows\simpo{\ucori{\recoco{\cardinal}}}$ is a Quillen adjunction.\end{corollary}

In fact a much stronger statement can be made, if we choose $\cardinal$ to be regular. Recall (e.g.\@ \cite{Du01} Def.\@ 2.1) that a model category $\caty$ is {\it combinatorial}, if it is cofibrantly generated and locally presentable.

\begin{theorem} Let $\cardinal\geq\aleph_0$ be a regular cardinal. The simplicial model structure on $\simpo{\recoco{\cardinal}}$ from Corollary \ref{ModelStructureBoundedCase} is combinatorial.\end{theorem}
\begin{proof} We claim that $\simpo{\recoco{\cardinal}}$ is locally $\cardinal$-presentable. According to Prop.\@ \ref{DoldKan} $\simpo{\recoco{\cardinal}}$ is equivalent to the category $\compo{\recoco{\cardinal}}$ of non-positively graded complexes in $\recoco{\cardinal}$. Thm.\@ \ref{LocalPresentability} tells us that each Banach space $\banachspace\in\recoco{\cardinal}$ is $\cardinal$-presentable. It follows easily that each bounded complex $\banachspace_{k}\rightarrow\banachspace_{k+1}\rightarrow\ldots\rightarrow\banachspace_{k+n}$ of Banach spaces is $\cardinal$-presentable in $\compo{\recoco{\cardinal}}$. 

\hide{%
Indeed, for any $m\in\mathbb Z_{\leq 0}$ the functor $\bornospace^\bullet\mapsto\bornospace^m$ has a right adjoint given by $\bornospace^m\rightarrow\bornospace^m$ put in degrees $m-1$ and $m$. Therefore colimits in $\compo{\recoco{\cardinal}}$ are computed degree-wise, and $\banachspace\rightarrow\banachspace$, put in any degree, is $\cardinal$-presentable for every Banach space $\banachspace\in\recoco{\cardinal}$. As finite limits commute with filtered colimits in $\sets$, $\underset{0\leq i\leq n}\bigoplus(\banachspace_{k+i}\rightarrow\banachspace_{k+i})$ is $\cardinal$-presentable (we put $\banachspace_{k+i}\rightarrow\banachspace_{k+i}$ in degrees $k+i$ and $k+i+1$). The obvious projection
	\begin{equation*}\underset{0\leq i\leq n}\bigoplus(\banachspace_{k+i}\rightarrow\banachspace_{k+i})\longrightarrow
	(\banachspace_{k}\rightarrow\banachspace_{k+1}\rightarrow\ldots\rightarrow\banachspace_{k+n})\end{equation*}
is surjective in each degree, hence a strict epimorphism. The kernel of this projection is again a finite complex of Banach spaces from $\recoco{\cardinal}$. Thus, using Lemma \ref{SmallProjectives}, we conclude that $\banachspace_{k}\rightarrow\banachspace_{k+1}\rightarrow\ldots\rightarrow\banachspace_{k+n}$ is $\cardinal$-presentable.}%

Let $\bornospace^\bullet\in\compo{\recoco{\cardinal}}$. For any $m\in\mathbb Z_{\leq 0}$ and any Banach subspace $\banachspace\subseteq\bornospace^m$ we can choose a subcomplex $\banachspace^\bullet\subseteq\bornospace^\bullet$ consisting of Banach subspaces, s.t.\@ $\banachspace^{<m}=0$ and $\banachspace^m=\banachspace$. Varying $m$, $\banachspace$ and all other choices, we have a filtered diagram of subcomplexes, whose colimit is $\bornospace^\bullet$. Indeed, in each degree it is just the dissection. So $\compo{\recoco{\cardinal}}$ is locally $\cardinal$-presentable.

The model structure on $\simpo{\recoco{\cardinal}}$ from Cor.\@ \ref{ModelStructureBoundedCase} is exactly the same as in \cite{ChristensenHovey02} Thm.\@ 6.3(**). Moreover, since Banach spaces are compact with respect to all monomorphisms we can use this theorem in loc.\@ cit.\@,\footnote{It was kindly communicated to us by D.Christensen, that Thm.\@ 6.3(**) in \cite{ChristensenHovey02} is valid only under the assumption that members of the projective class are compact relative to split monomorphisms. The arXiv version of \cite{ChristensenHovey02} is updated.}  and conclude that the model structure on $\simpo{\recoco{\cardinal}}$ is cofibrantly generated by 
	\begin{equation}\gencof:=\{\summable{\cardinal'}\otimes\boundary{n}\rightarrow\summable{\cardinal'}\otimes\simplex{n}\}_{n\geq 0},\end{equation}
	\begin{equation}\gentrico:=
	\{\summable{\cardinal'}\otimes\horn{n}{k}\rightarrow\summable{\cardinal'}\otimes\simplex{n}\}_{0<n\geq k\geq 0},\end{equation}
where $\cardinal'$ runs over all cardinals $<\cardinal$.\footnote{If $\cardinal=(\cardinal')^+$, it is enough to take this one $\cardinal'$.}\end{proof}%

Next we consider the monoidal and model structures together. Their interaction is governed by two axioms: the pushout product axiom and the monoid axiom (\cite{SchwedeShipley00} Def.\@ 3.1 and Def.\@ 3.3).

\begin{proposition}\label{GoodMonoidalStructure} Let $(\quabi,\ctensor)$ be a closed symmetric monoidal quasi-abelian category having enough projectives. Let $\{\pobj_i\}_{i\in I}$ be a generating projective class (Def.\@ \ref{GeneratingProjectives}), s.t.\@ $I$ is a set, each $\pobj_i$ is compact with respect to split monomorphisms, and $\{\pobj_i\}_{i\in I}$ is closed with respect to $\ctensor$. Then the category $\simpo{\quabi}$ of simplicial objects in $\quabi$ is a simplicial cofibrantly generated monoidal model category satisfying the monoid axiom.\end{proposition}
\begin{proof} By requiring that $I$ is a set and each $\pobj_i$ is compact with respect to split monomorphisms, we have made the model structure on $\simpo{\quabi}$, given by Prop.\@ \ref{ModelStructureInGeneral}, coincide with the model structure in \cite{ChristensenHovey02} Thm.\@ 6.3(**). Therefore it is cofibrantly generated.

We recall a well known fact.

\begin{lemma} Let $(\caty,\ctensor)$ be a closed symmetric monoidal category having all finite coproducts, and let $\simpo{\caty}$ be the category of simplicial objects in $\caty$. For any $\obj,\obj'\in\caty$ and any $\sset,\sset'\in\ssets$ we have in $\simpo{\caty}$
	\begin{equation*}(\obj\otimes\sset)\ctensor(\obj'\otimes\sset')\cong(\obj\ctensor\obj')\otimes(\sset\times\sset'),\end{equation*}
which is natural both in simplicial sets and objects of $\caty$.\end{lemma}
\hide{%
\begin{proof} By definition $(\obj\otimes\sset)_n:=\underset{s\in\sset_n}\coprod\,\obj$ (\cite{Qu67} \S II.1 Prop.\@ 2), and since $\ctensor$ commutes with direct sums we have $((\obj\otimes\sset)\ctensor(\obj'\otimes\sset'))_n=\underset{s\in\sset_n,s'\in\sset_n}\coprod(\obj\ctensor\obj')=((\obj\ctensor\obj')\otimes(\sset\times\sset'))_n$. Naturality is obvious.\end{proof}
}%
Given $\pobj_i,\pobj_j\in\quabi$ and cofibrations $\sset_i\rightarrow\sset'_i$, $\sset_j\rightarrow\sset'_j$ in $\ssets$ and using the previous lemma, we have a commutative diagram
	\begin{equation*}{\small\xymatrix{(\pobj_i\otimes\sset_i)\ctensor(\pobj_j\otimes\sset'_j)\underset{(\pobj_i\otimes\sset_i)\ctensor(\pobj_j\otimes\sset_j)}
	\coprod(\pobj_i\otimes\sset'_i)\ctensor(\pobj_j\otimes\sset_j)\ar[r]\ar[d]_\cong & (\pobj_i\otimes\sset'_i)\ctensor(\pobj_j\otimes\sset'_j)\ar[d]^\cong\\
	(\pobj_i\ctensor\pobj_j)\otimes(\sset_i\times\sset'_j\underset{\sset_i\times\sset_j}\coprod\sset'_i\times\sset_j)\ar[r] &
	(\pobj_i\ctensor\pobj_j)\otimes(\sset'_i\times\sset'_j).}}\end{equation*}
As $\sset_i\rightarrow\sset'_i$, $\sset_j\rightarrow\sset'_j$ are cofibrations, so is $\sset_i\times\sset'_j\underset{\sset_i\times\sset_j}\coprod\sset'_i\times\sset_j\rightarrow\sset'_i\times\sset'_j$, and if in addition one of the former is a trivial cofibration, so is the latter (\cite{GJ99} Prop.\@ 3.11). Since $\pobj_i\ctensor\pobj_j\in\{\pobj_i\}_{i\in I}$ it is cofibrant as a constant simplicial object in $\quabi$, therefore the bottom arrow in the above diagram is a cofibration or correspondingly a trivial cofibration (\cite{GJ99} Prop.\@ 3.4). So $\simpo{\quabi}$ is a monoidal model category (\cite{SchwedeShipley00} Lemma 3.5). 

According to \cite{SchwedeShipley00} Lemma 3.5, to prove that $\simpo{\quabi}$ satisfies the monoid axiom it is enough to show that transfinite compositions of co-base changes of elements of $\mathcal M:=\{\qobj\otimes\horn{n}{k}\rightarrow\qobj\otimes\simplex{n}\}_{0<n\geq k\geq 0}$ are weak equivalences (here $\qobj$ runs over all objects of $\quabi$). Using Prop.\@ \ref{DoldKan} we switch to $\compo{\quabi}$. It is immediate to see that for all $\qobj\in\quabi$, $n>0$ and each $n\geq k\geq 0$ the morphism in $\compo{\quabi}$ corresponding to $\qobj\otimes\horn{n}{k}\rightarrow\qobj\otimes\simplex{n}$ is the canonical inclusion 
	\begin{equation*}\qobj^\bullet\longrightarrow\qobj^\bullet\coprod(\qobj\overset=\rightarrow\qobj)[n],\end{equation*} 
where $(\qobj\overset=\rightarrow\qobj)[n]$ is concentrated in degrees $-n,-n+1$. Therefore co-base changes of elements of $\mathcal M$ are of the form $\qobj'^\bullet\rightarrow\qobj'^\bullet\coprod(\qobj\overset=\rightarrow\qobj)[n]$. Transfinite compositions of such morphisms are weak equivalences since 
$\underset{j\in J}\bigoplus(\qobj_j\overset=\rightarrow\qobj_j)[n_j]$ is acyclic for any $J$, $\{\qobj_j\}_{j\in J}$, $\{n_j\}_{j\in J}$.\end{proof}%

Now we can use \cite{SchwedeShipley00} Thm.\@ 4.1 to conclude the following.

\begin{theorem} Let $\cardinal\geq\aleph_0$ be regular, and let $\simo{\monoid}\in\simpo{\ucori{\recoco{\cardinal}}}$. The category of $\simo{\monoid}$-modules in $\simpo{\recoco{\cardinal}}$ is a cofibrantly generated monoidal model category satisfying the monoid axiom. A morphism is a fibration/weak equivalence, if it is a fibration/weak equivalence as a morphism in $\simpo{\recoco{\cardinal}}$.\end{theorem}

\begin{remark} For a simplicial commutative monoid $\simo{\monoid}$ in any closed monoidal quasi-abelian category $\quabi$, the monoidal structure on the category of $\simo{\monoid}$-modules is obtained by taking the co-equalizer of $\simo{\module}\ctensor\simo{\monoid}\ctensor\simo{\module'}\rightrightarrows\simo{\module}\ctensor\simo{\module'}$, where the arrows are the $\simo{\monoid}$-module structures on $\simo{\module}$ and $\simo{\module'}$ (e.g.\@ \cite{SchwedeShipley00} \S 4). 

Notice that $\modules{\simo{\monoid}}$ is a quasi-abelian category. Indeed, since $\ctensor$ is closed, the forgetful functor $\modules{\simo{\monoid}}\rightarrow\simpo{\quabi}$ has also a right adjoint, given by $\simo{\qobj}\mapsto\Hom_{\simpo{\quabi}}(\simo{\monoid},\simo{\qobj})$. Therefore $\modules{\simo{\monoid}}$ is additive, finitely complete and co-complete and pullbacks of strict epimorphisms/monomorphisms in $\modules{\simo{\monoid}}$  are stable under pullbacks/pushouts respectively.\end{remark}

Now we look at what happens, if we start with an adjunction between quasi-abelian categories. The first statement is obvious.

\begin{proposition}\label{JustAnAdjunction} Let $\adifu\colon\quabi\rightarrow\quabi'$ be a strictly exact functor between quasi-abelian categories with enough projectives. Suppose that $\adifu$ has a left adjoint $\adifu'$. Then $(\adifu',\adifu)$ induces a Quillen adjunction $\simo{\adifu'}\colon\simpo{\quabi'}\rightleftarrows\simpo{\quabi}\colon\simo{\adifu}$.\end{proposition}
\hide{\begin{proof} The fact that $(\simo{\adifu'},\simo{\adifu})$ is an adjunction is clear. We need to show that $\simo{\adifu}$ maps fibrations and trivial fibrations to morphisms of the same kind. Let $\simo{\qobj}\rightarrow\simo{\qobj'}$ be a fibration in $\simpo{\quabi}$. By definition this means that for every projective $\pobj\in\quabi$ the morphism of simplicial abelian groups
	\begin{equation*}\hom_{\quabi}(\pobj,\simo{\qobj})\longrightarrow\hom_{\quabi}(\pobj,\simo{\qobj'})\end{equation*}
is a fibration, i.e.\@ surjective in each simplicial dimension. Since this is supposed to hold for each projective $\pobj$, it is clear that $\simo{\qobj}\rightarrow\simo{\qobj'}$ is a strict epimorphism in each simplicial degree. Conversely, every morphism in $\simpo{\quabi}$ with this property is a fibration. Since by assumption $\adifu$ is strictly exact, it preserves strict epimorphisms, hence $\simo{\adifu}$ maps fibrations to fibrations.

Next we claim that $\adifu'$ maps projectives in $\quabi'$ to projectives in $\quabi$. Indeed, let $\pobj\in\quabi'$ be projective, and let $\qobj'\rightarrow\qobj''$ be a strict epimorphism in $\quabi$. Since $\adifu'$ is left adjoint to $\adifu$, we have a commutative diagram
	\begin{equation*}\xymatrix{\hom_{\quabi}(\adifu'(\pobj),\qobj')\ar[r]\ar[d]_\cong & \hom_{\quabi}(\adifu'(\pobj),\qobj'')\ar[d]^\cong\\
	\hom_{\quabi'}(\pobj,\adifu(\qobj'))\ar[r] & \hom_{\quabi'}(\pobj,\adifu(\qobj'')).}\end{equation*}
As $\adifu$ preserves strict epimorphisms and $\pobj$ is projective, the bottom row is surjective, hence so is the top row, i.e.\@ $\adifu'(\pobj)$ has to be projective in $\quabi$.

Now it is easy to see that $\simo{\adifu}$ maps weak equivalences to weak equivalences. Indeed, a morphism $\simo{\adifu}(\simo{\qobj}\rightarrow\simo{\qobj'})$ in $\simo{\quabi'}$ is a weak equivalence, iff its evaluation at each projective in $\quabi'$ is a weak equivalence of simplicial abelian groups. By adjunction this evaluation can be computed in $\quabi$, and since $\adifu'$ maps projectives to projectives, $\simo{\qobj}\rightarrow\simo{\qobj'}$ being a weak equivalence implies the same for $\simo{\adifu}(\simo{\qobj}\rightarrow\simo{\qobj'})$.\end{proof}}%

Sometimes such adjunctions are Quillen equivalences.

\begin{proposition}\label{QuillenEquivalence} Let $\adifu\colon\quabi\rightarrow\quabi'$ be a strictly exact fully faithful functor between quasi-abelian categories with enough projectives. Suppose that $\adifu$ has a left adjoint $\adifu'$, and for any $\qobj'\in\quabi'$ there is $\qobj\in\quabi$ and a strict epimorphism $\adifu(\qobj)\rightarrow\qobj'$. Then the Quillen adjunction $\simpo{\quabi'}\rightleftarrows\simpo{\quabi}$ from Prop.\@ \ref{JustAnAdjunction} is a Quillen equivalence.\end{proposition}
\begin{proof} First we show that $\adifu$ preserves projectives. Let $\pobj\in\quabi$ be projective, as $\quabi'$ has enough projectives there is a strict epimorphism $\pobj'\rightarrow\adifu(\pobj)$ where $\pobj'\in\quabi'$ is projective. By assumption $\exists\qobj\in\quabi$ and a strict epimorphism $\adifu(\qobj)\rightarrow\pobj'$ in $\quabi'$. Being a left adjoint $\adifu'$ preserves cokernels, thus 
	\begin{equation*}\qobj\cong\adifu'(\adifu(\qobj))\longrightarrow\adifu'(\pobj')\longrightarrow\adifu'(\adifu(\pobj))\cong\pobj\end{equation*}
are strict epimorphisms. Since $\pobj$ is projective there is a section $\pobj\rightarrow\qobj$, composing it with $\adifu(\qobj)\rightarrow\pobj'$ we see that $\adifu(\pobj)$ is a retract of $\pobj'$ and hence projective. From Cor.\@ \ref{ExplicitCofibrants} we conclude that $\simo{\adifu}$ preserves cofibrant objects.

By assumption every $\qobj'\in\quabi'$ is a quotient of some $\adifu(\qobj)$. Therefore, since $\adifu$ commutes with finite direct sums, $\{\adifu(\pobj)\}$, where $\pobj$ runs over all projective objects in $\quabi$, is a sufficiently large class of projectives in $\quabi'$ (Def.\@ \ref{GeneratingProjectives}). Applying Prop.\@ \ref{ProjectiveResolutions} and Cor.\@ \ref{ExplicitCofibrants} we see that every $\simo{\qobj'}\in\simpo{\quabi'}$ has a cofibrant resolution of the form $\simo{\adifu}(\simo{\pobj})$. Therefore, since $\adifu'\circ\adifu\cong\id_{\quabi}$, the two-out-of-three axiom implies that for any cofibrant $\simo{\qobj}\in\simpo{\quabi'}$ the natural $\simo{\qobj}\rightarrow\simo{\adifu}(\simo{\adifu'}(\simo{\qobj}))$ is a weak equivalence. By assumption the co-unit of $(\simo{\adifu'},\simo{\adifu})$ is invertible. Therefore $(\simo{\adifu'},\simo{\adifu})$ is a Quillen equivalence (\cite{Ho99} Prop.\@ 1.3.13).\end{proof}%

\begin{example} Let $\quabi$ be a quasi-abelian category with enough projectives. According to \cite{Schneiders99} Prop.\@ 1.3.24 the left abelian envelope $\abi$ of $\quabi$ has enough projectives.  Prop.\@ \ref{GoodEnvelopes}, \ref{QuillenEquivalence} tell us then that $\simpo{\abi}\rightleftarrows\simpo{\quabi}$ is a Quillen equivalence.\end{example}

\begin{example} The category $\coconbo$ of bornological spaces is a full reflective subcategory of the category $\indo{\banco}$ of Ind-Banach spaces (e.g.\@ \cite{Meyer2007} Thm.\@ 1.139), and every inductive system of Banach spaces is a quotient of an essentially monomorphic system (e.g.\@ \cite{BambozziBenBassat16} proof of Lemma 3.29). Therefore we have a Quillen equivalence $\simpo{\indo{\banco}}\rightleftarrows\simpo{\coconbo}$.\end{example}

\subsection{Cofibrant resolutions and flatness}\label{CofibrantResolutionsFlatness}

Let $\quabi$ be a quasi-abelian category with enough projectives. We have seen that $\simpo{\quabi}$ carries a model structure, defined by evaluating morphisms between objects in $\simpo{\quabi}$ at projective objects in $\quabi$ (Prop.\@ \ref{ModelStructureInGeneral}). We would like to have an explicit description of cofibrant objects and resolutions. This is easier done for cochain complexes, rather than simplicial objects.

\begin{definition}\label{ModelOnComplexes} Let $\quabi$ be a quasi-abelian category with enough projectives, let $\compo{\quabi}$ be the category of non-positively graded complexes in $\quabi$. A morphism $\mor\colon\como{Q}\rightarrow\como{Q'}$ is {\it a weak equivalence or fibration}, if for any projective $\pobj\in\quabi$ the morphism of complexes of abelian groups $\hom_{\quabi}(\pobj,\como{Q})\rightarrow\hom_{\quabi}(\pobj,\como{Q'})$ is a weak equivalence or a fibration respectively.\end{definition} 
Recall that we have the normalized complex functor $\normacho\colon\simpo{\quabi}\rightarrow\compo{\quabi}$. According to Prop.\@ \ref{DoldKan} each component of $\simo{Q}$ is a finite direct sum of degenerations of components of $\normach{\simo{Q}}$. Since finite direct sums in $\quabi$ equal finite direct products, the Dold--Kan correspondence for abelian groups implies that $\hom_{\quabi}(P,\simo{Q})\rightarrow\hom_{\quabi}(P,\simo{Q'})$ is a weak equivalence of simplicial abelian groups, if and only if $\hom_{\quabi}(P,\normach{\simo{Q}})\rightarrow\hom_{\quabi}(P,\normach{\simo{Q'}})$ is a weak equivalence of complexes of abelian groups.

Any morphism $\simo{Q}\rightarrow\simo{Q'}$ maps $\normach{\simo{Q}}\rightarrow\normach{\simo{Q'}}$ and degenerations to degenerations. Hence a morphism in $\simpo{\quabi}$ is a fibration, iff its image in $\compo{\quabi}$ is a fibration. Altogether we have the following.

\begin{corollary}\label{ModelStructuresEqual} Fibrations and weak equivalences from Def.\@ \ref{ModelOnComplexes} define a model structure on $\compo{\quabi}$, and the equivalence from Prop.\@ \ref{DoldKan} identifies this model structure with the one on $\simpo{\quabi}$.\end{corollary}

The following proposition is standard.

\begin{proposition}\label{ProjectiveResolutions} Let $\quabi$ be a quasi-abelian category with enough projectives, and let $\{\pobj_i\}_{i\in I}\subseteq\quabi$ be a sufficiently large class of projectives in $\quabi$ (Def.\@ \ref{GeneratingProjectives}). Let $\qobj^\bullet:=\ldots\rightarrow\qobj^{-1}\rightarrow\qobj^0$ be a complex in $\quabi$. There is a trivial fibration $\pobj^\bullet\rightarrow\qobj^\bullet$, s.t.\@ $\pobj^k\in\{\pobj_i\}$ for each $k\leq 0$.\end{proposition}
\begin{proof} In the context of quasi-abelian categories the standard procedure of adding and killing cycles is based on the following obvious lemma.

\begin{lemma}\label{FillingShortExact} Let $\quabi$ be a quasi-abelian category with enough projectives, and let $\{\pobj_i\}_{i\in I}\subseteq\quabi$ be a sufficiently large class of projectives in $\quabi$ (Def.\@ \ref{GeneratingProjectives}). Let $0\rightarrow\qobj_1\rightarrow\qobj_2\rightarrow\qobj_3\rightarrow 0$ be a strictly exact sequence, and let $\mor\colon\qobj\rightarrow\qobj_3$ be any morphism in $\quabi$. Then there is a commutative diagram
	\begin{equation*}\xymatrix{0\ar[r]\ar[d] & \K\ar[r]\ar[d]_{\emor} & \pobj\ar[r]\ar[d] & \qobj\ar[r]\ar[d]_{\mor} & 0\ar[d]\\
	0\ar[r] & \qobj_1\ar[r] & \qobj_2\ar[r] & \qobj_3\ar[r] & 0,}\end{equation*}
where the upper row is strictly exact, $\pobj\in\{\pobj_i\}$, and $\emor$ is a strict epimorphism.\end{lemma}
\begin{proof} Let $\qobj':=\qobj_2\underset{\qobj_3}\times\qobj$. Since $\qobj_2\rightarrow\qobj_3$ is a strict epimorphism and $\quabi$ is quasi-abelian, also $\qobj'\rightarrow\qobj$ is a strict epimorphism. Since $\{\pobj_i\}_{i\in I}$ is a sufficiently large class of projectives in $\quabi$ we can find $\pobj'\in\{\pobj_i\}$ and a strict epimorphism $\pobj'\rightarrow\qobj'$. Clearly $\pobj'\rightarrow\qobj$ is a strict epimorphism. Let $\K'\rightarrow\pobj'$ be the kernel of this epimorphism. We have a commutative diagram 
	\begin{equation*}\xymatrix{0\ar[r]\ar[d] & \K'\ar[r]\ar[d] & \pobj'\ar[r]\ar[d] & \qobj\ar[r]\ar[d]_{\mor} & 0\ar[d]\\
	0\ar[r] & \qobj_1\ar[r] & \qobj_2\ar[r] & \qobj_3\ar[r] & 0,}\end{equation*}
where the upper row is strictly exact, but $\K'\rightarrow\qobj_1$ does not have to be a strict epimorphism. We choose a strict epimorphism $\pobj''\rightarrow\qobj_1$ with $\pobj$ among $\{\pobj_i\}_{i\in I}$ and define $\K:=\K'\oplus\pobj''$, $\pobj:=\pobj'\oplus\pobj''$. Mapping $\pobj''\rightarrow 0\rightarrow\qobj$, $\pobj''\rightarrow\qobj_1\rightarrow\qobj_2$, we obtain a commutative diagram
	\begin{equation*}\xymatrix{0\ar[r]\ar[d] & \K\ar[r]\ar[d] & \pobj\ar[r]\ar[d] & \qobj\ar[r]\ar[d]_{\mor} & 0\ar[d]\\
	0\ar[r] & \qobj_1\ar[r] & \qobj_2\ar[r] & \qobj_3\ar[r] & 0.}\end{equation*}
Projection on a summand is a strict epimorphism, hence $\pobj\rightarrow\qobj$ is a strict epimorphism. The map $\pobj''\rightarrow\qobj_1$ factors through $\K\rightarrow\qobj_1$, hence the latter is a strict epimorphism as well (e.g.\@ \cite{Schneiders99}, Prop.\@ 1.1.8). Finally any map to $\pobj'\oplus\pobj''$ that composes to $0$ with $\pobj'\oplus\pobj''\rightarrow\qobj$ factors uniquely through $\K$, i.e.\@ the top arrow is strictly exact.\end{proof}%

Lemma \ref{FillingShortExact} immediately gives us the following add/kill procedure.

\begin{lemma}\label{AddKill} Let $\quabi$ be a quasi-abelian category with enough projectives, and let $\{\pobj_i\}_{i\in I}\subseteq\quabi$ be a sufficiently large class of projectives in $\quabi$ (Def.\@ \ref{GeneratingProjectives}). Consider a commutative diagram in $\quabi$
	\begin{equation*}\xymatrix{ & \pobj^{k+1}\ar[d]\ar[r]^{d'_{k+1}} & \pobj^{k+2}\ar[d]\\
	\qobj^k\ar[r]^{d_k} & \qobj^{k+1}\ar[r]^{d_{k+1}} & \qobj^{k+2},}\end{equation*}
where the bottom row is a null sequence, the vertical arrows and $\kernel{d'_{k+1}}\rightarrow\kernel{d_{k+1}}$ are strict epimorphisms. Then we can extend this diagram to 
	\begin{equation}\label{OneToTheLeft}\xymatrix{\pobj^k\ar[r]^{d'_k}\ar[d] & \pobj^{k+1}\ar[r]^{d'_{k+1}}\ar[d] & \pobj^{k+2}\ar[d]\\
	\qobj^k\ar[r]^{d_k} & \qobj^{k+1}\ar[r]^{d_{k+1}} & \qobj^{k+2},}\end{equation}
s.t.\@ $\pobj^k\in\{\pobj_i\}$, the top row is a null sequence, all vertical arrows are strict epimorphisms, $\forall i\in I$ $\Hom_\quabi(\pobj_i,\pobj^\bullet)\rightarrow\hom_{\quabi}(\pobj_i,\qobj^\bullet)$ is a weak equivalence in degree $k+1$, and $\kernel{d'_{k}}\rightarrow\kernel{d_{k}}$ is a strict epimorphism.\end{lemma}
\hide{%
\begin{proof} Let $\qobj'$ be the kernel of $\kernel{d'_{k+1}}\rightarrow\kernel{d_{k+1}}$. Consider
	\begin{equation*}\xymatrix{& & & \qobj^{k}\ar[d]^{d_k} & 
	\\ 0\ar[r] & \qobj'\ar[r] & \kernel{d'_{k+1}}\ar[r] & \kernel{d_{k+1}}\ar[r] & 0.}\end{equation*}
Using Lemma \ref{FillingShortExact} we can complete this diagram to
	\begin{equation*}\xymatrix{0\ar[d]\ar[r] & \K\ar[d]_{\emor}\ar[r] & \pobj^k\ar[d]\ar[r] & \qobj^{k}\ar[d]^{d_k}\ar[r] & 0\ar[d] 
	\\ 0\ar[r] & \qobj'\ar[r] & \kernel{d'_{k+1}}\ar[r] & \kernel{d_{k+1}}\ar[r] & 0,}\end{equation*}
s.t.\@ $\pobj^k\in\{\pobj_i\}$, the upper row is strictly exact, and $\emor$ is a strict epimorphism. Defining $d'_{k}\colon\pobj^k\rightarrow\kernel{d'_{k+1}}\rightarrow\pobj^{k+1}$ we get (\ref{OneToTheLeft}), where the top row is a null sequence and all vertical morphisms are strict epimorphisms. Since $\yoneda{\pobj_i}{\quabi}$ takes strictly exact sequences to exact sequences and $\emor$ is a strict epimorphism, we immediately see that $\Hom_\quabi(\pobj_i,\pobj^\bullet)\rightarrow\hom_{\quabi}(\pobj_i,\qobj^\bullet)$ is a weak equivalence in degree $k+1$. The same long exact sequence also shows that $\kernel{d'_{k}}\rightarrow\kernel{d_{k}}$ is a strict epimorphism.\end{proof}}%

Repeatedly applying Lemma \ref{AddKill} we obtain a proof of the proposition. \hide{
We denote $\qobj^1:=0$, and define $\pobj^1:=0$. Now suppose for a given $k\in\mathbb Z_{\leq 0}$ we have built a commutative diagram
	\begin{equation*}\xymatrix{& & \pobj^{k+1}\ar[r]^{d'_{k+1}}\ar[d] & \ldots\ar[r]^{d'_0} & \pobj^1\ar[d]\\
	\ar[r]^{d_{k-1}} & \qobj^k\ar[r]^{d_k} & \qobj^{k+1}\ar[r]^{d_{k+1}} & \ldots\ar[r]^{d_0} & \qobj^1,}\end{equation*}
s.t.\@ the top row is a complex, all vertical arrows are strict epimorphisms, $\forall i\in I$ $\hom_\quabi(\pobj_i,\pobj^\bullet)\rightarrow\hom_{\quabi}(\pobj_i,\qobj^\bullet)$ is a weak equivalence in degrees $>k+1$, and $\kernel{d'_{k+1}}\rightarrow\kernel{d_{k+1}}$ is a strict epimorphism. Using Lemma \ref{AddKill} we extend this diagram to 
	\begin{equation*}\xymatrix{& \pobj^k\ar[r]^{d'_k}\ar[d] & \pobj^{k+1}\ar[r]^{d'_{k+1}}\ar[d] & \ldots\ar[r]^{d'_0} & \pobj^1\ar[d]\\
	\ar[r]^{d_{k-1}} & \qobj^k\ar[r]^{d_k} & \qobj^{k+1}\ar[r]^{d_{k+1}} & \ldots\ar[r]^{d_0} & \qobj^1,}\end{equation*}
s.t.\@ $\pobj^k\in\{\pobj_i\}$, the top row is a complex, all vertical arrows are strict epimorphisms, $\forall i\in I$ $\hom_\quabi(\pobj_i,\pobj^\bullet)\rightarrow\hom_{\quabi}(\pobj_i,\qobj^\bullet)$ is a weak equivalence in degrees $>k$, and $\kernel{d'_{k}}\rightarrow\kernel{d_{k}}$ is a strict epimorphism.}\end{proof}%

Just as in the abelian case, the standard argument gives the following.

\begin{corollary}\label{ExplicitCofibrants} Let $\quabi$ be a quasi-abelian category with enough projectives, and let $\{\pobj_i\}_{i\in I}\subseteq\quabi$ be a sufficiently large class of projectives (Def.\@ \ref{GeneratingProjectives}). An object $\como{\qobj}\in\compo{\quabi}$ is cofibrant, iff each $\qobj^k$ is a retract of some $\pobj_k\in\{\pobj_i\}$. An object $\simo{\qobj}\in\simpo{\quabi}$ cofibrant, iff $\simo{\qobj}\cong\simcho{\como{\qobj}}$ for a cofibrant $\como{\qobj}\in\compo{\quabi}$.\end{corollary}
\hide{
\begin{proof} According to Prop.\@ \ref{ProjectiveResolutions} for each $\como{\qobj}\in\compo{\quabi}$ there is a trivial fibration $\como{\pobj}\rightarrow\como{\qobj}$, where components of $\como{\pobj}$ belong to $\{\pobj_i\}$. Therefore it is enough to show that $\como{\pobj}$ is cofibrant, if each $\pobj^k$ is projective. Let $\como{\emor}\colon\como{\qobj}\rightarrow\como{\qobj'}$ be a trivial fibration, and let $\como{\mor}\colon\como{\pobj}\rightarrow\como{\qobj'}$ be any morphism.

Denoting $\qobj^{\geq 1}=\qobj'^{\geq 1}:=0$, suppose that we have constructed a lifting $\como{\mor'}\colon\como{\pobj}\rightarrow\como{\qobj}$ of $\como{\mor}$ in degrees $>k$ for some $k\leq 0$. As $\pobj^k$ is projective and $\emor^k$ is a strict epimorphism, we can choose a lifting $\widetilde{\mor}^{k}\colon\pobj^k\rightarrow\qobj^k$ of $\mor^k$. Denote $\delta^k:=\mor'^{k+1}\circ d_{\como{\pobj}}^k-d_{\como{\qobj'}}^k\circ\widetilde{\mor}^{k}$. Clearly $\emor^{k+1}\circ\delta^k=0$. We claim that $d_{\como{\qobj'}}^{k+1}\circ\delta^k=0$. Indeed, this immediately follows from ${\mor'}^{k+2}\circ d_{\como{\pobj}}^{k+1}-d_{\como{\qobj'}}^{k+1}\circ{\mor'}^{k+1}=0$.

Therefore $\delta^k$ factors through the kernel of $\K^{k+1}\rightarrow\K^{k+2}$, where $\como{\K}$ is the kernel of $\como{\emor}$. Since $\como{\emor}$ is a weak equivalence, $\hom_{\quabi}(\pobj^k,\como{\K})$ is an acyclic complex of abelian groups, i.e.\@ $\delta^k$ lifts to $\epsilon^k\colon\pobj^k\rightarrow\K^k$. Defining ${\mor'}^k:=\widetilde{\mor}^k+\epsilon^k$ we are done.\end{proof}}%

Similarly to Prop.\@ \ref{ProjectiveResolutions} we have a characterization of cofibrant objects in the category $\modules{\simo{\monoid}}$ of $\simo{\monoid}$-modules in $\simpo{\quabi}$ for a commutative monoid $\simo{\monoid}\in\simpo{\quabi}$. First we recall a standard definition (e.g.\@ \cite{GoerssSch04} Prop.\@ 4.2(3)).

\begin{definition}\label{DefAlmoFree} Let $(\monoca,\monote)$ be a closed symmetric monoidal finitely cocomplete category, and let $\simo{\monoid}\in\simpo{\monoca}$ be a simplicial commutative monoid in $\monoca$. An $\simo{\monoid}$-module $\simo{\module}$ is {\it almost freely generated by} $\{\obj_n\}_{n\geq 0}\subseteq\monoca$, if $\forall n\geq 0$
	\begin{equation}\label{AlmostFreeModule}\module_n\cong\underset{m\leq n}\coprod\,\underset{\finor{n}\twoheadrightarrow\finor{m}}\coprod
	\monoid_n\monote\obj_m\end{equation}
as $\monoid_n$-modules, and for any $\sigma\colon\finor{k}\twoheadrightarrow\finor{n}$ the corresponding map $\module_n\rightarrow\module_k$ is given on each summand in (\ref{AlmostFreeModule}) by composing $\finor{k}\rightarrow\finor{n}\rightarrow\finor{m}$ and mapping 
	\begin{equation*}\sigma^*\ctensor\id_{\obj_m}\colon\monoid_n\monote\obj_m\longrightarrow\monoid_k\monote\obj_m.\end{equation*}
\end{definition}
It is well known that, if we start with an abelian category with enough projectives, we can always construct resolutions of modules over simplicial monoids, that are almost freely generated by projectives. The same is true in the quasi-abelian case with essentially the same proof.

\begin{lemma}\label{AlmostFreeResolution} Let $(\quabi,\ctensor)$ be a closed symmetric monoidal quasi-abelian category with enough projectives, and let $\{\pobj_i\}_{i\in I}\subseteq\quabi$ be a sufficiently large class of projectives (Def.\@ \ref{GeneratingProjectives}). Let $\simo{\monoid}$ be a commutative monoid in $\simpo{\quabi}$, and let $\simo{\module}\in\modules{\simo{\monoid}}$. There is a sequence $\{\pobj_n\}_{n\geq 0}\subseteq\{\pobj_i\}_{i\in I}$ and a trivial fibration $\simo{\module'}\rightarrow\simo{\module}$, s.t.\@ $\simo{\module'}$ is almost freely generated by $\{\pobj_n\}_{n\geq 0}$ over $\simo{\monoid}$.\end{lemma}
\begin{proof} This is just another straightforward application of Lemma \ref{AddKill}. \hide{%
Let $n\in\mathbb Z_{\geq 0}$ and suppose we have constructed $\simo{\module^{(n)}}\rightarrow\simo{\module}$ s.t.\@\begin{itemize}
\item[1.] $\simo{\module^{(n)}}$ is almost freely generated over $\simo{\monoid}$ by $\{\pobj_0,\ldots,\pobj_n,0,0,\ldots\}$,
\item[2.] $\forall i\in I,\forall k<n$ $\pi_k(\hom_{\quabi}(\pobj_i,\simo{\module^{(n)}}))\overset\cong\longrightarrow\pi_k(\hom_{\quabi}(\pobj_i,\simo{\module}))$,
\item[3.] $\forall k\leq n$ $\module^{(n)}_k\rightarrow\module_k$ is a strict epimorphism,
\item[4.] let $\K^{(n)}_n\rightarrow\module^{(n)}_n$, $\K_n\rightarrow\module_n$ be the kernels of $\normach{\simo{\module^{(n)}}}^{-n}\rightarrow\normach{\simo{\module^{(n)}}}^{-n+1}$, $\normach{\simo{\module}}^{-n}\rightarrow\normach{\simo{\module}}^{-n+1}$, then $\K^{(n)}_n\rightarrow\K_n$ is a strict epimorphism.\end{itemize}
Then we can extend this to $\simo{\module^{(n+1)}}$ having similar properties. Indeed, $\normach{\simo{\module^{(n)}}}\rightarrow\normach{\simo{\module}}$ is a morphism of complexes that in degrees $\geq-n$ consists of strict epimorphisms and the morphism on $-n$-cocycles is also a strict epimorphism. 

Applying Lemma \ref{AddKill} we can find $(-1)^{n+1}\boundaryo_{n+1}\colon\pobj_{n+1}\rightarrow\K^{(n)}_n$, $\pobj_{n+1}\rightarrow\normach{\simo{\module}}^{-n-1}$ with the latter being a strict epimorphism, s.t.\@ $\pobj_{n+1}\in\{\pobj_i\}_{i\in I}$, 
	\begin{equation*}\xymatrix{\ar[r] & \normach{\simo{\module^{(n)}}}^{-n-1}\oplus\pobj_{n+1}\ar[r]\ar[d] & \normach{\simo{\module^{(n)}}}^{-n}\ar[r]\ar[d] &
	\normach{\simo{\module^{(n)}}}^{-n+1}\ar[r]\ar[d] &\\
	\ar[r] & \normach{\simo{\module}}^{-n-1}\ar[r] & \normach{\simo{\module}}^{-n}\ar[r] & \normach{\simo{\module}}^{-n+1}\ar[r] & }\end{equation*}
is a weak equivalence in degrees $\geq-n$ and the morphism of $-n-1$-cocycles is a strict epimorphism. Now we define $\forall k\geq 0$
	\begin{equation*}\module^{(n+1)}_k:=\module^{(n)}_k\oplus
	\left(\underset{\finor{k}\twoheadrightarrow\finor{n+1}}\bigoplus\monoid_k\ctensor\pobj_{n+1}\right)\end{equation*}
with the degeneration maps as in Def.\@ \ref{DefAlmoFree}. For $k=n+1$ the only non-zero new boundary is $\monoid_{n+1}\ctensor\pobj_{n+1}\overset{\boundaryo_{n+1}\ctensor\boundaryo_{n+1}}\longrightarrow\monoid_n\ctensor\K^{(n)}_n\rightarrow\K^{(n)}_n\hookrightarrow\module_n^{(n)}$. This and simplicial identities define new boundary maps for each $k>n+1$. 

It is clear that $\simo{\module^{(n+1)}}$ is almost freely generated by $\{\pobj_0,\ldots,\pobj_{n+1},0,\ldots\}$. Being an intersection of kernels, $\normach{\simo{\module}}^{-n-1}$ inherits an $\monoid_{n+1}$-module structure from $\module_{n+1}$. Therefore we have a well defined $\monoid_{n+1}\ctensor\pobj_{n+1}\rightarrow\module_{n+1}$, that extends to a morphism of $\simo{\monoid}$-modules $\simo{\module^{(n+1)}}\rightarrow\simo{\module}$.

Let $\pobj\in\{\pobj_i\}_{i\in I}$. From the assumptions on $\simo{\module^{(n)}}$ we know that 
	\begin{equation}\label{BuildUp}\pi_k(\hom_{\quabi}(\pobj,\simo{\module^{(n+1)}}))\longrightarrow
	\pi_k(\hom_{\quabi}(\pobj,\simo{\module}))\end{equation}
is an isomorphism for $k<n$ and surjective for $k=n$, since $\K^{(n)}_n\rightarrow\K_n$ is a strict epimorphism. By construction an element in the kernel of (\ref{BuildUp}) for $k=n$ lifts to $\pobj\rightarrow\pobj_{n+1}$, and hence also to $\monoid_{n+1}\ctensor\pobj_{n+1}$. Thus (\ref{BuildUp}) is an isomorphism for $k\geq n$.

From $\pobj_{n+1}\rightarrow\normach{\simo{\module}}^{-n-1}$ being a strict epimorphism it follows immediately that $\module^{(n)}_{n+1}\rightarrow\module_{n+1}$ is also a strict epimorphism, and similarly for $\K^{(n+1)}_{n+1}\rightarrow\K_{n+1}$. Thus $\simo{\module^{(n+1)}}$ satisfies all $4$ conditions above.

The beginning of this inductive construction is with a choice of a strict epimorphism $\pobj_0\rightarrow\module_0$, s.t.\@ $\pobj_0\in\{\pobj_i\}_{i\in I}$. Since $\ctensor$ is closed the composite $\monoid_0\ctensor\pobj_0\rightarrow\monoid_0\ctensor\module_0\rightarrow\module_0$ is a strict epimorphism (the last arrow has a section given by the unit, so it is a strict epimorphism as well).}%
\end{proof}%

We have an immediate consequence.

\begin{corollary}\label{CofibrantRetract} Let $(\quabi,\ctensor)$, $\{\pobj_i\}_{i\in I}$ be as in Prop.\@ \ref{GoodMonoidalStructure}, and suppose every $\simo{\qobj}\in\simpo{\quabi}$ is small. Let $\simo{\monoid}\in\simpo{\quabi}$ be a commutative monoid. Any cofibrant $\simo{\monoid}$-module is a retract of an $\simo{\monoid}$-module that is almost freely generated by direct sums of elements of $\{\pobj_i\}_{i\in I}$. \end{corollary}
\begin{proof} By requiring that $(\quabi,\ctensor)$, $\{\pobj_i\}_{i\in I}$ are as in Prop.\@ \ref{GoodMonoidalStructure} and each $\simo{\qobj}$ in $\simpo{\quabi}$ is small we ensure that the category of $\simo{\monoid}$-modules is a cofibrantly generated monoidal model category (\cite{SchwedeShipley00} Thm.\@ 4.1). From Lemma \ref{AlmostFreeResolution} we know then that every cofibrant $\simo{\monoid}$-module can be resolved by an almost free $\simo{\monoid}$-module generated by direct sums of $\pobj_i$'s. Since a trivial fibration into a cofibrant object always has a right inverse, we are done.\end{proof}

Using this characterization of cofibrant modules we can prove flatness.

\begin{definition} Let $(\quabi,\ctensor)$ be a closed monoidal quasi-abelian category. An object $\qobj\in\quabi$ is {\it flat} if the functor $-\ctensor\qobj\colon\quabi\rightarrow\quabi$ is strictly exact.\end{definition}
Since $\ctensor$ is closed, $-\ctensor\qobj$ preserves strict epimorphisms, and since $\quabi$ is additive, it also preserves finite direct products. Therefore $\qobj$ is flat, if and only if $-\ctensor\qobj$ preserves kernels, i.e.\@ strict monomorphisms. 

\begin{example}\label{FlatBanach} For any $\cardinal$ $\summable{\cardinal}$ is flat in $\banco$ (\cite{Grothendieck66} \S I.2.2 Cor.\@ 3). Then for any set $I$ $\underset{i\in I}\bigoplus\summable{\cardinal_i}$ is flat in $\coconbo$. Indeed, since $\ctensor$ commutes with direct sums $-\ctensor\left(\underset{i\in I}\bigoplus\summable{\cardinal_i}\right)$ takes strict monomorphisms to direct sums of strict monomorphisms. Such direct sum is injective, its image is closed (any bornologically convergent sequence in an arbitrary direct sum is contained in the sum of a finite subfamily) and the bornology on the domain is induced from the one on the target (both are direct sum bornologies). \end{example}
In proving flatness it is useful to have the following obvious facts.

\begin{proposition}\label{CofibrantFlat} Let $(\quabi,\ctensor)$ be a closed monoidal quasi-abelian category, and let $\qobj\in\quabi$ be flat. Then every retract of $\qobj$ is flat.

If every $\simo{\qobj}\in\simpo{\quabi}$ is small and there is a generating projective class $\{\pobj_i\}_{i\in I}\subseteq\quabi$ as in Prop.\@ \ref{GoodMonoidalStructure}, s.t.\@ all small direct sums of elements of $\{\pobj_i\}_{i\in I}$ are flat, then $\forall\simo{\monoid}\in\simpo{\ucori{\quabi}}$ any cofibrant $\simo{\monoid}$-module is flat.
\end{proposition}
\begin{proof} Let $\qobj'$ be a retract of $\qobj$, and let $\qobj_1\rightarrow\qobj_2$ be a strict monomorphism. We have a commutative diagram
	\begin{equation*}\xymatrix{\qobj'\ctensor\qobj_1\ar[r]\ar[d] & \qobj\ctensor\qobj_1\ar[r]\ar[d] & \qobj'\ctensor\qobj_1\ar[d]\\
	\qobj'\ctensor\qobj_2\ar[r] & \qobj\ctensor\qobj_2\ar[r] & \qobj'\ctensor\qobj_2,}\end{equation*}
where both horizontal compositions are identities and the middle vertical arrow is a strict monomorphism. Since $\qobj'\ctensor\qobj_1\rightarrow\qobj\ctensor\qobj_1$ factors the identity, it is a strict monomorphism (e.g.\@ \cite{Schneiders99} Prop.\@ 1.1.8). Then $\qobj'\ctensor\qobj_1\rightarrow\qobj\ctensor\qobj_2$ is a strict monomorphism and hence so is $\qobj'\ctensor\qobj_1\rightarrow\qobj'\ctensor\qobj_2$ (loc.\@ cit.\@).

According to Cor.\@ \ref{CofibrantRetract} every cofibrant $\simo{\monoid}$-module is a retract of an $\simo{\monoid}$-module that is almost freely generated by direct sums of elements of $\{\pobj_i\}_{i\in I}$. Therefore it is enough to prove that every such almost free $\simo{\monoid}$-module $\simo{\module}$ is flat. Let $\simo{\module'}\rightarrow\simo{\module''}$ be a strict monomorphism in $\modules{\simo{\monoid}}$. By assumption $\forall n\geq 0$ $\module_n\cong\monoid_n\ctensor\qobj_n$ as $\monoid_n$-modules, where $\qobj_n$ is flat as an object of $(\quabi,\ctensor)$. Then $\module'_n\underset{\monoid_n}\ctensor\module_n\rightarrow\module''_n\underset{\monoid_n}\ctensor\module_n$ is isomorphic to $\module'_n\ctensor\qobj_n\rightarrow\module''\ctensor\qobj_n$, which is a strict monomorphism in $\quabi$. As this is also a morphism of $\monoid_n$-modules, the quotient has an induced structure of an $\monoid_n$-module, implying that this morphism is a strict monomorphism in $\modules{\monoid_n}$ as well.\end{proof}%

\begin{example}\label{FlatnessOfBornologicalModules} Let $\cardinal\geq\cardinal_0$ be a regular, and let $\simo{\monoid}\in\simpo{\ucori{\recoco{\cardinal}}}$. Using Ex.\@ \ref{FlatBanach} and Prop.\@ \ref{CofibrantFlat} we see that every cofibrant $\simo{\monoid}$-module is flat.\end{example}
One of the uses of flat objects is the following standard fact.

\begin{proposition}\label{FlatnessOfCofibrants} Let $(\quabi,\ctensor)$, $\{\pobj_i\}_{i\in I}$ be as in Prop.\@ \ref{GoodMonoidalStructure}, and suppose every $\simo{\qobj}\in\simpo{\quabi}$ is small. Let $\simo{\monoid}$ be a simplicial commutative monoid in $\quabi$, and let $\simo{\module}$ be a flat $\simo{\monoid}$-module. Then $-\underset{\simo{\monoid}}\ctensor\simo{\module}$ preserves weak equivalences.\end{proposition}
\begin{proof} It is enough to show that $-\underset{\simo{\monoid}}\ctensor\simo{\module}$ preserves acyclic objects. \hide{%
Indeed, by definition every object in $\modules{\simo{\monoid}}$ is fibrant, therefore, according to Ken Brown's lemma (e.g.\@ \cite{Ho99} Lemma 1.1.12) it is enough to show that $-\underset{\simo{\monoid}}\ctensor\simo{\module}$ maps trivial fibrations to weak equivalences. Given a trivial fibration and taking its kernel we obtain a short strict exact sequence. Since $\simo{\monoid}$ is assumed to be flat, $-\underset{\simo{\monoid}}\ctensor\simo{\module}$ preserves short strict exact sequences. Therefore it is enough to show that $-\underset{\simo{\monoid}}\ctensor\simo{\module}$ preserves acyclic objects in $\modules{\simo{\monoid}}$.}%
Given an acyclic $\simo{\module'}$ and using induction on simplicial dimension $n\geq 0$ we have a strict short exact sequence of $\simo{\monoid}$-modules
	\begin{equation*}0\longrightarrow\simo{\module^+}\longrightarrow\simo{\module'}\longrightarrow\simo{\module^{n,n+1}}\longrightarrow 0,\end{equation*}
where $\module'_{<n}=0$, $\module^+_{<n+1}=0$ and $\simo{\module^{n,n+1}}$ is generated as an $\simo{\monoid}$-module by components in degrees $n$, $n+1$. \hide{%
Let $n\in\mathbb Z_{\geq 0}$ and let $\simo{\module'}$ be an acyclic object in $\modules{\simo{\monoid}}$, s.t.\@ $\module'_k=0$ for all $k<n$. For each $m\geq 0$ let $\module^+_m\rightarrow\module'_m$ be the intersection of kernels of $\{\module'_m\rightarrow\module'_{n}\}_{\finor{n}\hookrightarrow\finor{m}}$.\footnote{Alternatively we could define $\module^+_m$ as the intersection of kernels of maps corresponding to all $\finor{n}\rightarrow\finor{m}$. This would give us the same result, since $\module'_{<n}=0$.} As a morphism in $\quabi$ $\module^+_m\rightarrow\module'_m$ is isomorphic to
	\begin{equation}\label{HigherTail}\left(\underset{m\twoheadrightarrow n+1}\bigoplus\K_{n+1}\right)\oplus
	\left(\underset{m\geq k\geq n+2}\bigoplus\,\underset{m\twoheadrightarrow k}\bigoplus\normach{\simo{\module'}}^{-k}\right)
	\longrightarrow\module'_m,\end{equation}
where $\K_{n+1}\rightarrow\module'_{n+1}$ is the kernel of $\normach{\simo{\module'}}^{-n-1}\rightarrow\normach{\simo{\module'}}^{-n}$. Being an intersection of kernels, $\module^+_m\rightarrow\module'_m$ is a strict monomorphism of $\monoid_m$-modules $\forall m\geq 0$. Let $\phi\colon\finor{m_1}\rightarrow\finor{m_2}$ and consider the corresponding composition $\phi^*\colon\module^+_{m_2}\hookrightarrow\module'_{m_2}\rightarrow\module'_{m_1}$. Since $\forall\psi\colon\finor{n}\rightarrow\finor{m_1}$ the composition $(\phi\circ\psi)^*\colon\module^+_{m_2}\hookrightarrow\module'_{m_2}\rightarrow\module'_n$ is $0$, $\phi^*$ factors through $\module^+_{m_1}$, i.e.\@ we have a morphism of $\simo{\monoid}$-modules $\simo{\module^+}\rightarrow\simo{\module'}$. Let $\simo{\module^{n,n+1}}$ be the quotient of this morphism. }%
Moreover $\forall m\geq 0$ as objects of $\quabi$
	\begin{equation}\label{ModuleDecomposition}\module^{n,n+1}_m\cong\left(\underset{\finor{m}\twoheadrightarrow\finor{n}}\bigoplus\module'_n\right)
	\oplus\left(\underset{\finor{m}\twoheadrightarrow\finor{n+1}}\bigoplus\module''_n\right),\end{equation}
where $\module''_n$ is another copy of $\module'_n$, identified with $\module'_n$ by $\boun{n+1}\colon\module'_{n+1}\rightarrow\module'_n$. It is not difficult to see that (\ref{ModuleDecomposition}) is a decomposition into $\monoid_m$-submodules, and altogether this is a realization of $\simo{\module^{n,n+1}}$ as a simplicial $\monoid_n$-module. \hide{%
Indeed, $\underset{0\leq i\leq n}\sum(-1)^{i}\boun{i}\colon\module^{n,n+1}_{n+1}\rightarrow\module^{n,n+1}_n$ in a monomorphism when restricted to $\module''_n$ and $0$ when restricted to the other summands, similarly for any $0\leq k<n$ $\underset{0\leq i\leq k}\sum(-1)^{i}\boun{i}\colon\module^{n,n+1}_{n+1}\rightarrow\module^{n,n+1}_n$ is a monomorphism on $\degen{k}(\module'_n)$ and $0$ everywhere else. Therefore, using Prop.\@ \ref{ManipulationWithKernels} we see that in dimension $n+1$ (\ref{ModuleDecomposition}) is a decomposition into $\monoid_{n+1}$-modules. Moreover, all these $\monoid_{n+1}$-modules are isomorphic. Indeed, for every $0\leq i\leq n$ $\boun{i}$, $\boun{i+1}$ identify $\degen{i}(\module'_n)$ with $\module'_n$ as $\monoid_{n+1}$-modules. Therefore the two structures of an $\monoid_{n+1}$-module on $\module'_n$ given by $\boun{i},\boun{i+1}\colon\monoid_{n+1}\rightarrow\monoid_n$ are isomorphic to the one on $\degen{i}(\module'_n)$. Finally $\boun{n+1}$ identifies the $\monoid_{n+1}$-module structure on $\module''_n$ with the one on $\module'_n$ given by $\boun{n+1}\colon\monoid_{n+1}\rightarrow\monoid_n$.

Now we use induction on $l$ to show that (\ref{ModuleDecomposition}) is a decomposition into a direct sum of $\monoid_{n+l}$-modules also for $m=n+l$. For any $\phi\colon\finor{m+l}\twoheadrightarrow\finor n$ let $k_\phi\geq 0$ be the smallest number, s.t.\@ $\phi(k_\phi)=\phi(k_\phi+1)$. Let $\phi':=\phi\circ\degen{k_\phi}$, and let $\pi_{\phi'}\colon\module^{n,n+1}_{n+l-1}\rightarrow(\phi')^*(\module')$ be the projection on a summand. Then $\pi_{\phi'}\circ\underset{0\leq i\leq k_\phi}\sum(-1)^i\boun{i}$ is a monomorphism on $\phi^*(\module')$ and is $0$ on every other summand of $\module^{n,n+1}_{n+l}$. 

In a similar way, starting with $\psi\colon\finor{n+l}\twoheadrightarrow\finor{n+1}$ and using a projection on a degeneration of $\module''_n$ in $\module^{n,n+1}_{k+l-1}$, we see that we can use Prop.\@ \ref{ManipulationWithKernels} and realize each summand in $\module^{n,n+1}_{n+l}$ as an intersection of kernels. Altogether we see that the structure of an $\simo{\monoid}$-module on $\simo{\module^{n,n+1}}$ is induced from that of a simplicial $\monoid_n$-module.}%
It is clear that $\simo{\module^{n,n+1}}$ is acyclic. We would like to show that $\simo{\module^{n,n+1}}\underset{\simo{\monoid}}\ctensor\simo{\module}$ is acyclic as well. We notice that we have a homotopy $\simo{\module^{n,n+1}}\otimes\simplex{1}\rightarrow\simo{\module^{n,n+1}}$ from the identity to the $0$ morphism. \hide{%
Indeed, consider $\simo{\module^{n,n+1}}\otimes\simplex{1}$. By definition $\forall k\geq 0$
	\begin{equation*}(\simo{\module^{n,n+1}}\otimes\simplex{1})_k=\underset{\finor{k}\rightarrow\finor{1}}\bigoplus\module^{n,n+1}_k.\end{equation*}
For any $\phi\colon\finor{k}\rightarrow\finor{1}$ let $\finor{m}\hookrightarrow\finor{k}$ be the inclusion of the fiber over $0\in\finor{1}$, and let $\boundaryo_\phi\colon\module^{n,n+1}_k\rightarrow\module^{n,n+1}_m$ be the corresponding boundary. When the fiber is empty we define $\boundaryo_\phi\colon\module^{n,n+1}\rightarrow 0$. Let $\K'_\phi\rightarrow\module^{n,n+1}_k$ be the kernel of $\boundaryo_\phi$, this is a strict monomorphism of $\monoid_k$-modules. Altogether we have $\forall k\geq 0$
	\begin{equation*}\K'_k:=\underset{\phi\colon\finor{k}\rightarrow\finor{1}}\bigoplus\K'_\phi\longrightarrow
	(\simo{\module^{n,n+1}}\otimes\simplex{1})_k.\end{equation*}
We claim $\simo{\K'}\rightarrow\simo{\module^{n,n+1}}\otimes\simplex{1}$ a strict monomorphism of $\simo{\monoid}$-modules. To prove this we just need to show that $\simo{\K'}$ is stable with respect to simplicial structure maps on $\simo{\module^{n,n+1}}\otimes\simplex{1}$. For any $\psi\colon\finor{k}\rightarrow\finor{k'}$, $\phi\colon\finor{k'}\rightarrow\finor{1}$ we have the natural map $\finor{m}\rightarrow\finor{m'}$ where $\finor{m}$, $\finor{m'}$ are the fibers of $\phi\circ\psi$ and respectively $\phi$ over $0\in\finor{1}$. This immediately implies that $\psi^*$ maps $\K'_\phi\rightarrow\K'_{\phi\circ\psi}$. We denote
	\begin{equation*}\simo{\overline{\module}^{n,n+1}}:=(\simo{\module^{n,n+1}}\otimes\simplex{1})/\simo{\K'}.\end{equation*}
From the construction it is clear that $\forall k\geq 0$ as $\monoid_k$-modules
	\begin{equation*}\overline{\module}^{n,n+1}_k\cong\underset{n\leq i\leq k}\bigoplus\module^{n,n+1}_i,\end{equation*}
where the $\monoid_k$-module structure on $\module^{n,n+1}_i$ comes from $\monoid_k\rightarrow\monoid_i$ given by $\finor{i}\hookrightarrow\finor{k}$ mapping $j\mapsto j$. The index $i$ corresponds to $\sigma_i\colon\finor{k}\rightarrow\finor{1}$ that maps $\finor{i}\mapsto 0$ and the rest to $1$. 

Considered as a simplicial $\monoid_n$-module, $\simo{\overline{\module}^{n,n+1}}$ has only $4$ non-degenerate summands: $\module'_n$ in dimension $n$ and in dimension $n+1$ indexed by $\sigma_n$; and $\module''_n$ in dimension $n+1$ and in dimension $n+2$ indexed by $\sigma_{n+1}$. Each one of the summands lies in $\normach{\simo{\overline{\module}^{n,n+1}}}$, and the differential identifies the copies of $\module'_n$ and $\module''_n$. Therefore we have a well defined morphism in $\simpo{\modules{\monoid_n}}$: 
	\begin{equation}\label{TheEnd}\simo{\overline{\module}^{n,n+1}}\longrightarrow\simo{\module^{n,n+1}},\end{equation}
$\module'_n\overset{=}\rightarrow\module'_n$, $\module''_n\overset{=}\rightarrow\module''$ in dimensions $n$ and $n+1$; and $\module'_n\overset\cong\rightarrow\module''_n$, $\module''_n\overset\cong\rightarrow\degen{n+1}(\module''_n)$ in dimensions $n+1$ and $n+2$. Since the $\simo{\monoid}$-module structure on $\simo{\module^{n,n+1}}$ is induced from the simplicial $\monoid_n$-module structure, (\ref{TheEnd}) is a morphism of $\simo{\monoid}$-modules as well.}%
Therefore, since 
	\begin{equation*}\simo{\module}\ctensor\left(\simo{\module^{n,n+1}}\otimes\simplex{1}\right)\cong
	\left(\simo{\module}\ctensor\simo{\module^{n,n+1}}\right)\otimes\simplex{1}\end{equation*} 
we conclude that $\simo{\module}\ctensor\simo{\module^{n,n+1}}$ is acyclic. Using flatness of $\simo{\module}$ and continuing inductively on $n$ we are done.\end{proof}%

Putting together Ex.\@ \ref{FlatnessOfBornologicalModules} and Prop.\@ \ref{FlatnessOfCofibrants} we have the following.

\begin{theorem}\label{InvarianceOfModules} Let $\cardinal\geq\aleph_0$ be a regular cardinal. For any weak equivalence $\simo{\monoid}\rightarrow\simo{\monoid'}$ of commutative simplicial monoids in $(\recoco{\cardinal},\ctensor)$ the canonical adjunction $\modules{\simo{\monoid}}\rightleftarrows\modules{\simo{\monoid'}}$ is a Quillen equivalence. \end{theorem}
\begin{proof} \cite{SchwedeShipley00} Theorem 4.3.\end{proof}%

\section{Bornological rings of $\cinfty$-functions}

In this section we embed the category of simplicial $\cinfty$-rings into the category $\simpo{\ucori{\coubo}}$ of simplicial commutative monoids in the category $\coubo$ of locally separable complete bornological spaces.

\subsection{The pre-compact bornology}\label{PrecompactBornology}

The $\mathbb R$-vector space $\cinfty(\mathbb R^n)$ of smooth functions on $\mathbb R^n$ is a nuclear Fr\'echet space (e.g.\@ \cite{HM81} \S 5.3 Thm.\@ 2). A set $U\subseteq\cinfty(\mathbb R^n)$ is a neighbourhood of $0$, if there are $r,\epsilon\in\mathbb R_{>0}$, $\rqt\in\mathbb R^n$, $m\in\mathbb Z_{\geq 0}$ s.t.\@ $\wball{\rqt}{r}{\epsilon}{m}\subseteq U$, where
	\begin{equation*}\wball{\rqt}{r}{\epsilon}{m}:=
	\{f\in\cinfty(\mathbb R^n)\,|\,\underset{\norm{\rpt-\rqt}{}\leq r}\sup\norm{\jet{f}{m}(\rpt)}{}<\epsilon\}\end{equation*}
and $\norm{\jet{f}{m}(\rpt)}{}:=\underset{\norm{\mindex{k}}{}\leq m}\sum\norm{\diffo{f}{\mindex{k}}(\rpt)}{}$. Given a Fr\'echet space $\fpace$ we denote by $\cpt{\fpace}$ the {\it pre-compact complete bornological space on $\fpace$} (e.g.\@ \cite{Meyer2007} Def.\@ 1.15). Clearly $\cpt{\fpace}$ is locally separable (e.g.\@ \cite{Meyer2007} Prop.\@ 1.163).

\begin{proposition} (E.g.\@ \cite{Meyer2007} Thm.\@ 1.29, Example 1.76, Thm.\@ 1.87) The functor $\cpto$ is a full and faithful, fully exact,\footnote{{\it Full exactness} here means that $\cpto$ preserves and reflects strictly exact sequences.} symmetric monoidal functor from the category $(\fcati,\ctensor)$ of Fr\'echet spaces and completed projective tensor product to the category $(\coubo,\ctensor)$ of locally separable complete bornological spaces.\end{proposition}

\begin{remark} The other natural choice of a bornology -- the von Neumann bornology -- produces the same result for $\cinfty(\mathbb R^n)$, since for nuclear Fr\'echet spaces the pre-compact and von Neumann bornologies coincide (e.g.\@ \cite{BBK15} Lemma 3.67).\end{remark}

A linear map $\cinfty(\mathbb R^n)\rightarrow\cinfty(\mathbb R^m)$ is a $\cinfty$-morphism, iff it is a continuous morphism of commutative Fr\'echet algebras (\cite{KKM87} Thm.\@ 2.4). Moreover, in $\fcati$ we have $\cinfty(\mathbb R^n)\ctensor\cinfty(\mathbb R^m)\cong\cinfty(\mathbb R^{n+m})$ as algebras (e.g.\@ \cite{GS03} Thm.\@ 6.12). Therefore, denoting by $\frifig$ the category of free finitely generated $\cinfty$-rings, we see that $\frifig\subset\ucori{\fcati}$ as a full subcategory, closed with respect to finite coproducts. 

\begin{corollary} The functor $\cpto\colon\frifig\rightarrow\ucori{\coubo}$ is full, faithful and preserves finite coproducts.\end{corollary}

For an arbitrary set $\set$ we define $\cinfty(\mathbb R^{\set})$ as a co-limit of $\{\cinfty(\mathbb R^{\set'})\}_{\set'\subseteq\set}$ (computed in the category of $\mathbb R$-vector spaces) where $\set'\subseteq\set$ runs over finite subsets. Equivalently $\cinfty(\mathbb R^{\set})$ can be defined as the set of functions $\mathbb R^{\set}\rightarrow\mathbb R$, that factor through a projection on an $\mathbb R^n\subseteq\mathbb R^{\set}$ and a smooth $\mathbb R^n\rightarrow\mathbb R$. 

Being a filtered union of $\cpt{\cinfty(\mathbb R^n)}$'s, $\cinfty(\mathbb R^{\set})$ is a complete bornological space, with a subset being bounded, if it is a bounded subset of some $\cinfty(\mathbb R^n)$. Clearly this is an object of $\coubo$, and we denote it by $\cpt{\cinfty(\mathbb R^{\set})}$.

\smallskip

Since $\cpt{\cinfty(\mathbb R^m)}\ctensor\cpt{\cinfty(\mathbb R^n)}\cong\cpt{\cinfty(\mathbb R^{m+n})}$, and $\ctensor$ commutes with inductive limits, we see that 
	\begin{equation*}\cpt{\cinfty(\mathbb R^{\set_1})}\ctensor\cpt{\cinfty(\mathbb R^{\set_2})}\cong\cpt{\cinfty(\mathbb R^{\set_1\sqcup\set_2})}.\end{equation*}
The product operation maps $\cinfty(\mathbb R^{n_1})\ctensor\cinfty(\mathbb R^{n_2})\rightarrow\cinfty(\mathbb R^{n_1+n_2})$ for all finite $n_1,n_2\leq\set$, hence it is bounded and $\cpt{\cinfty(\mathbb R^{\set})}$ is a monoid in $\coubo$.

\smallskip

By definition $\forall\set'\subseteq\set$ $\cpt{\cinfty(\mathbb R^{\set'})}\rightarrow\cpt{\cinfty(\mathbb R^{\set})}$ is injective, with the left inverse given by restriction. This restriction morphism is bounded, since a subset of $\cinfty(\mathbb R^{\set'})$ is bounded, if it is a bounded subset of some $\cinfty(\mathbb R^n)\subseteq\cinfty(\mathbb R^{\set'})$, and similarly for $\cinfty(\mathbb R^{\set})$. Therefore it is enough to consider the left inverse of $\cpt{\cinfty(\mathbb R^m)}\rightarrow\cpt{\cinfty(\mathbb R^{m+n})}$. In particular $\forall\set'\subseteq\set$ $\cpt{\cinfty(\mathbb R^{\set'})}\rightarrow\cpt{\cinfty(\mathbb R^{\set})}$ is a strict monomorphism.
\hide{%
Indeed, the property $f\notin\cinfty(\mathbb R^m)$ means there is at least one $\rpt\in\mathbb R^{m+n}\setminus\mathbb R^m$, s.t.\@ $f(\rpt)\neq f(\rqt)$, where $\rqt$ is the projection of $\rpt$ on $\mathbb R^m$. Consider $f+\wball{\rqt}{r}{\epsilon}{0}$, where $r=\|\rqt-\rpt\|$, $\epsilon<\frac{1}{2}|f(\rpt)-f(\rqt)|$. Clearly $\forall g\in f+\wball{\rqt}{r}{\epsilon}{0}$ we have $g(\rpt)\neq g(\rqt)$.}%

Since $\cpt{\cinfty(\mathbb R^{\set'})}$ is bornologically closed in $\cpt{\cinfty(\mathbb R^{\set})}$, for any $\cmor\colon\cpt{\cinfty(\mathbb R^{\set_1})}\rightarrow\cpt{\cinfty(\mathbb R^{\set_2})}$ in $\ucori{\coubo}$, and any finite $\set'_1\subseteq\set_1$ there is a finite $\set'_2\subseteq\set_2$, s.t.\@ $\cmor(\cpt{\cinfty(\mathbb R^{\set'_1})})\subseteq\cpt{\cinfty(\mathbb R^{\set'_2})}$.

Indeed, we can always find $\set'_2$ s.t.\@ $\cmor(\set'_1)\subseteq\cpt{\cinfty(\mathbb R^{\set'_2})}$. Then all polynomials in $\set'_1$ are mapped to $\cpt{\cinfty(\mathbb R^{\set'_2})}$. Every element of $\cinfty(\mathbb R^{\set'_1})$ is a limit of a bornological sequence consisting of polynomials in $\set'_1$ (e.g.\@ \cite{Meyer2007} Thm.\@ 1.36), and we use the fact that $\cpt{\cinfty(\mathbb R^{\set'_2})}$ is closed in $\cpt{\cinfty(\mathbb R^{\set_2})}$. 

Since morphisms $\cpt{\cinfty(\mathbb R^m)}\rightarrow\cpt{\cinfty(\mathbb R^n)}$ in $\ucori{\coubo}$ coincide with $\cinfty$-morphisms, we see that $\cmor$ is a diagram of $\cinfty$-morphisms between free finitely generated $\cinfty$-rings. By definition this makes $\cmor$ into a $\cinfty$-morphisms $\cinfty(\mathbb R^{\set_1})\rightarrow\cinfty(\mathbb R^{\set_2})$. Altogether we have the following.

\begin{proposition}\label{PreservationOfCoproducts} Let $\freci$ be the category of all free $\cinfty$-rings. We have 
	\begin{equation}\label{Embedding}\cpto\colon\freci\longrightarrow\ucori{\coubo}\end{equation} 
that is full, faithful and preserves finite coproducts.\end{proposition}

\subsection{Model structure}\label{ModelStructureFunctions}

Now we look at the behaviour of (\ref{Embedding}) with respect to the model structure on simplicial $\cinfty$-rings. Here a morphism is a fibration or a weak equivalence if the underlying morphism of simplicial abelian groups is respectively a fibration or a weak equivalence (\cite{Qu67} \S II.4). As always we denote by $\simpo{\fcati}$ the category of simplicial objects in the category $\fcati$ of Fr\'echet spaces. We start with the following fact.

\begin{proposition}\label{SurjectiveFrechet} Let $\phi\colon\simo{\fpace}\rightarrow\simo{\fpace'}$ be a surjective morphism in $\simpo{\fcati}$. Then $\cpt{\simo{\fpace}}\rightarrow\cpt{\simo{\fpace'}}$ is a weak equivalence in $\simpo{\coubo}$, if and only if the underlying map of simplicial $\mathbb R$-spaces is a weak equivalence.\end{proposition}
\begin{proof} By definition $\cpt{\simo{\fpace}}\rightarrow\cpt{\simo{\fpace'}}$ is a weak equivalence, if it gives a weak equivalence, when evaluated at each $\summable{\cardinal}$, $\cardinal\leq\aleph_0$, and in particular for $\cardinal=1$. Therefore the only if direction is clear.

A morphism of Fr\'echet spaces is a strict epimorphism, iff it is surjective. Therefore it is enough to prove that $\cpt{\kernel{\phi}}$ is acyclic. Switching to $\compo{\coubo}$, it is enough then to show that complexes of Fr\'echet spaces, that are exact as complexes of $\mathbb R$-spaces, are strictly exact in $\coubo$. This is true because a morphism of Fr\'echet spaces is strict, iff its image is closed.\end{proof}%

Let $\simpo{\freci}$ be the category of simplicial objects in $\freci$, and let $\alfreci\subset\simpo{\freci}$ be the full subcategory consisting of almost free simplicial $\cinfty$-rings, i.e.\@ those whose degenerations map generators to generators. 

\begin{proposition}\label{TheEmbedding} The functor $\cpto\colon\alfreci\rightarrow\simpo{\ucori{\coubo}}$ is full, faithful and preserves finite coproducts. Moreover, a morphism $\simo{\cmor}$ in $\alfreci$ is a weak equivalence or a trivial fibration, if and only if $\cpt{\simo{\cmor}}$ is a weak equivalence or respectively a trivial fibration in $\simpo{\ucori{\coubo}}$.\end{proposition}
\begin{proof} As $\cpto\colon\freci\rightarrow\ucori{\coubo}$ is fully faithful and preserves finite coproducts, the same is true for simplicial diagrams. All objects of $\alfreci$ are cofibrant simplicial $\cinfty$-rings, hence trivial fibrations in $\alfreci$ have right inverses and $\cpto$ maps trivial fibrations in $\alfreci$ to fibrations in $\simpo{\ucori{\coubo}}$.

If $\cpt{\simo{\cmor}}$ is a weak equivalence in $\simpo{\ucori{\coubo}}$, its evaluation at $\mathbb R$ is a weak equivalence of simplicial abelian groups, i.e.\@ $\simo{\cmor}$ is a weak equivalence in $\alfreci$. It remains to show that $\cpto$ preserves weak equivalences. 

\begin{lemma}\label{HomotopyInBornology} For any $\simo{\cring}\in\alfreci$ let $\sigma\colon\simo{\cring}\otimes\simplex{1}\rightarrow\simo{\cring}$ be given by $\simplex{1}\rightarrow\simplex{0}$. Then $\cpt{\sigma}$ is a weak equivalence in $\simpo{\ucori{\coubo}}$.\end{lemma}
\begin{proof} Since $\sigma$ has a right inverse, it is enough to show that $\kernel{\cpt{\sigma}}$ is acyclic. By definition $\forall k\in\mathbb Z_{\geq 0}$ $\cring_k\cong\cinfty(\mathbb R^{\set_k})$ for some set $\set_k$, and the sets $\{\set_k\}_{k\geq 0}$ are stable under degenerations in $\simo{\cring}$. Choose $k\geq 0$, and let $\set'_k\subseteq\set_k$ be any finite subset. 

We can find a finite $\set'_{k-1}\subseteq\set_{k-1}$ s.t.\@ all boundaries of $\set'_k$ lie in $\cinfty(\mathbb R^{\set'_{k-1}})$. Continuing like this we find $\{\cinfty(\mathbb R^{\set'_m})\subseteq\cinfty(\mathbb R^{\set_m})\}_{m\leq k}$ that are stable under boundaries. Taking all degenerations of $\{\set'_{m}\}_{m\leq k}$ we obtain $\simo{\cring'}\subseteq\simo{\cring}$, s.t.\@ $\simo{\cring'}$ is an almost free $\cinfty$-ring finitely generated in each simplicial dimension, and $\cinfty(\mathbb R^{\set'_k})\subseteq\cring'_k$. Therefore $\simo{\cring}$ is a filtered union of almost free simplicial $\cinfty$-rings, finitely generated in each simplicial dimension.

By definition $\forall k\in\mathbb Z_{\geq 0}$ $(\simo{\cring}\otimes\simplex{1})_k\cong\cring_k^{\ctensor^{k+2}}$, with the simplicial structure maps given by boundaries and degenerations in $\simplex{1}$. Since taking the $\ctensor$-product of free $\cinfty$-rings corresponds to taking disjoint unions of the generators, it is clear that $\simo{\cring}\otimes\simplex{1}$ is a filtered union of $\{\simo{\cring'}\otimes\simplex{1}\}$, where $\simo{\cring'}\subseteq\simo{\cring}$ runs over all almost free simplicial $\cinfty$-subrings of $\simo{\cring}$, finitely generated in each simplicial dimension.

\smallskip

Let $k\in\mathbb Z_{\geq 0}$ and let $\alpha\colon\samcou\rightarrow\kernel{\cpt{\sigma}}_k$ be a morphism in $\coubo$. Since Banach spaces are compact relative to monomorphisms there is an almost free $\simo{\cring'}\subseteq\simo{\cring}$, s.t.\@ $\simo{\cring'}$ is finitely generated in each dimension, and $\alpha$ factors through $\cpt{\cring'_k\otimes\simplex{1}}\cap\kernel{\cpt{\sigma}}=\kernel{\cpt{\sigma'}}$, where $\sigma'\colon\simo{\cring'}\otimes\simplex{1}\rightarrow\simo{\cring'}$ is given by $\simplex{1}\rightarrow\simplex{0}$. 

So in proving that $\cpt{\sigma}$ is a weak equivalence, we can assume that $\simo{\cring}$ is finitely generated in each simplicial dimension. Then $\cpt{\sigma}$ is a surjective morphism of simplicial Fr\'echet spaces, and it is a weak equivalence, if and only if the underlying morphism of simplicial abelian groups is a weak equivalence (Prop.\@ \ref{SurjectiveFrechet}). Thus $\cpt{\sigma}$ is a weak equivalence.\end{proof}%

\smallskip

Now we prove that $\cpto$ preserves all weak equivalences. As each object of $\alfreci$ is a cofibrant simplicial $\cinfty$-ring, a weak equivalence $\simo{\cmor}\colon\simo{\cring}\rightarrow\simo{\cring'}$ has a quasi-inverse $\simo{\cmor'}\colon\simo{\cring'}\rightarrow\simo{\cring}$. If we prove that $\cpt{\simo{\cmor}\circ\simo{\cmor'}}$, $\cpt{\simo{\cmor'}\circ\simo{\cmor}}$ are homotopic to identities, it would imply that $\simo{\cmor}$, $\simo{\cmor'}$ are weak equivalences.

So let $\simo{\cmor},\simo{\cmor'}\colon\simo{\cring}\rightrightarrows\simo{\cring'}$ be two homotopic morphisms in $\alfreci$, we need to show that $\cpt{\simo{\cmor}}$, $\cpt{\simo{\cmor'}}$ are homotopic in $\simpo{\ucori{\coubo}}$. As $\simo{\cring}$ is a cofibrant $\cinfty$-ring, there is $\simo{\cmor''}\colon\simo{\cring}\otimes\simplex{1}\rightarrow\simo{\cring'}$, s.t.\@ $\simo{\cmor}=\simo{\cmor''}\circ\iota_1$, $\simo{\cmor'}=\simo{\cmor''}\circ\iota_2$, where $\iota_1,\iota_2\colon\simo{\cring}\rightarrow\simo{\cring}\otimes\simplex{1}$ are the canonical inclusions. We have seen (Lemma \ref{HomotopyInBornology}) that $\cpt{\sigma}$ is a weak equivalence, and since $\cpt{\sigma}\circ\cpt{\iota_1}=\cpt{\sigma}\circ\cpt{\iota_2}=\id_{\cpt{\simo{\cring}}}$, so are $\cpt{\iota_1}$, $\cpt{\iota_2}$, and moreover $\cpt{\iota_1}=\cpt{\iota_2}$ in the homotopy category of $\simpo{\ucori{\coubo}}$. Therefore $\cpt{\simo{\cmor}}=\cpt{\simo{\cmor'}}$ in this homotopy category.\end{proof}%

\subsection{Quasi-coherent sheaves}\label{QuasiCoherentSheaves}

In this section we use the functor $\cpto$ to bring categories of modules from simplicial bornological rings to simplicial $\cinfty$-rings. Proposition \ref{TheEmbedding} together with Thm.\@ \ref{InvarianceOfModules} immediately imply the following.

\begin{theorem} Let $\resolution{\simo{\cring}}\twoheadrightarrow\simo{\cring}$ be any functorial almost free resolution of simplicial $\cinfty$-rings. Associating $\simo{\cring}\mapsto\modules{\cpt{\resolution{\simo{\cring}}}}$ we get a (pseudo-)functor\footnote{This assignment is only a pseudo-functor due to associativity isomorphisms for $\ctensor$ on $\coubo$. We will ignore this distinction and treat this construction as a functor.} from $\simpo{\allrings}$ to the category of model categories and Quillen adjunctions. Weak equivalences of simplicial $\cinfty$-rings are mapped to Quillen equivalences. 

Let $\resolution{\simo{\cring''}}\leftarrow\resolution{\simo{\cring}}\rightarrow\resolution{\simo{\cring'}}$ be almost free morphisms of almost free simplicial $\cinfty$-rings. In the co-cartesian square
	\begin{equation}\label{BaseChangeDiagram}\xymatrix{\resolution{\simo{\cring}}\ar[d]_{\phi}\ar[rr]^{\psi} && \resolution{\simo{\cring'}}\ar[d]^{\phi'}\\
	\resolution{\simo{\cring''}}\ar[rr]_{\psi'} && \resolution{\simo{\cring'''}}}\end{equation}
the two functors 
	\begin{equation}\label{BaseChange}\xymatrix{\modules{\cpt{\resolution{\simo{\cring''}}}}
	\ar@<-.5ex>[r]_{\phi'^*\circ\psi'_*}\ar@<.5ex>[r]^{\psi_*\circ\phi^*} & 
	\modules{\cpt{\resolution{\simo{\cring'}}}}}\end{equation} 
are naturally equivalent. \end{theorem}
\begin{proof} \hide{%
Theorem \ref{InvarianceOfModules} tells us that weak equivalences of simplicial commutative bornological monoids translate to Quillen equivalences of the corresponding categories of simplicial modules. On the other hand, Prop.\@ \ref{TheEmbedding} tells us that $\cpto$ maps weak equivalences between almost free simplicial $\cinfty$-rings to weak equivalences of bornological monoids. Hence the first claim of the theorem.}%

The colimit in (\ref{BaseChangeDiagram}) is computed separately in each simplicial dimension, and since the morphisms are almost free, $\forall k\geq 0$ $\cring'''_k$ is obtained by taking a triple coproduct of $\cinfty$-rings. From Prop.\@ \ref{PreservationOfCoproducts} we know that $\cpto$ preserves finite coproducts. Therefore $\cpt{\resolution{\simo{\cring'''}}}\cong\cpt{\resolution{\simo{\cring''}}}\underset{\cpt{\resolution{\simo{\cring}}}}\ctensor\cpt{\resolution{\simo{\cring'}}}$ and the natural equivalence in (\ref{BaseChange}) is given by the associativity natural equivalence for $\ctensor$ on $\coubo$. \hide{%
Indeed, let $\monoid''\leftarrow\monoid\rightarrow\monoid'$ be morphisms of commutative monoids in $\coubo$, and let $\module$ be an $\monoid''$-module. Then $(\monoid''\underset{\monoid}\ctensor\monoid')\underset{\monoid''}\ctensor\module$ is obtained from $\monoid''\ctensor\monoid'\ctensor\module$ by dividing by the action of $\monoid$ on $\monoid''$ and $\monoid'$ and then by the action of $\monoid''$ on the result and on $\module$. Alternatively we can first divide by the action of $\monoid''$ on $\monoid''$ and on $\module$ to obtain $\monoid'\ctensor\module$ and then divide by the action of $\monoid$. This alternative way gives us the pullback to $\monoid$-modules followed by the push-forward to $\monoid'$-modules.}%
\end{proof}%

\smallskip

If we have a diagram $\simo{\cring''}\leftarrow\simo{\cring}\rightarrow\simo{\cring'}$ of simplicial $\cinfty$-rings where neither the rings nor the morphisms are almost free, we can compute the homotopy colimit of this diagram by resolving it into (\ref{BaseChangeDiagram}) and taking the usual colimit. The natural equivalence (\ref{BaseChange}) will correspond then to a natural {\it weak equivalence} between the {\it derived functors}, i.e.\@ functors obtained by pre-composing with cofibrant resolutions.

\smallskip

Now we look closer at the categories of modules in some special cases. In the previous section we had to make an additional effort to deal with $\cinfty$-rings that are not finitely generated, the reason being that such rings are not Fr\'echet. Sometimes we can do with Fr\'echet $\cinfty$-rings only.

\begin{definition} A simplicial $\cinfty$-ring $\simo{\cring}$ is {\it of finite presentation}, if there is a surjective weak equivalence of simplicial $\cinfty$-rings $\resolution{\simo{\cring}}\rightarrow\simo{\cring}$, s.t.\@ $\resolution{\simo{\cring}}$ is almost free and finitely generated in each simplicial dimension. \end{definition}
By requiring that each $\resolution{\cring_k}$ is a finitely generated $\cinfty$-ring we ensure that it is a Fr\'echet space.
\begin{proposition}\label{FrechetFinitePresentation} Let $\simo{\cring}$ be a simplicial $\cinfty$-ring of finite presentation, s.t.\@ $\forall k\geq 0$ $\cring_k$ is a finitely generated Fr\'echet $\cinfty$-ring, i.e.\@ $\cring_k\cong\cinfty(\mathbb R^{n_k})/\ideal_k$ for some $n_k\geq 0$ and a closed ideal $\ideal_k$. Let $\rho\colon\resolution{\simo{\cring}}\rightarrow\simo{\cring}$ be an almost free resolution, then $\cpt{\rho}\colon\cpt{\resolution{\simo{\cring}}}\rightarrow\cpt{\simo{\cring}}$ is a weak equivalence.\end{proposition}
\begin{proof} From Prop.\@ \ref{TheEmbedding} we know that any two almost free resolutions of $\simo{\cring}$ are mapped by $\cpto$ to weakly equivalent simplicial bornological rings. Therefore we can assume that $\resolution{\simo{\cring}}$ is finitely generated in each simplicial dimension. Then $\cpt{\rho}$ comes from a surjective morphism of simplicial Fr\'echet spaces, and hence it is a weak equivalence (Prop.\@ \ref{SurjectiveFrechet}).\end{proof}%

\smallskip

Proposition \ref{FrechetFinitePresentation} implies that, given a simplicial $\cinfty$-ring of finite presentation $\simo{\cring}$, s.t.\@ each $\cring_k$ is Fr\'echet, instead of going through an almost free resolution $\resolution{\simo{\cring}}$ we can take $\modules{\cpt{\simo{\cring}}}$ and obtain a Quillen equivalent model category. As before, given a co-cartesian diagram in $\simpo{\allrings}$
	\begin{equation*}\xymatrix{\simo{\cring}\ar[d]_{f_2}\ar[rr]^{f_1} && \simo{\cring'}\ar[d]^{\widetilde{f}_2}\\ \simo{\cring''}\ar[rr]_{\widetilde{f}_1} && \simo{\cring'''}}\end{equation*}
of such simplicial $\cinfty$-rings, associativity of $\ctensor$ in $\coubo$ gives a natural equivalence of functors $(f_1)_*\circ f_2^*\rightarrow\widetilde{f}_2^*\circ(\widetilde{f}_1)_*$. In particular this applies to constant simplicial Fr\'echet $\cinfty$-rings of finite presentation, i.e.\@ ordinary Fr\'echet $\cinfty$-rings that have almost free resolutions with finitely many generators in each simplicial dimension.

Let $\cring$ be such a $\cinfty$-ring, and denote by $\simo{\cring}$ the corresponding constant simplicial $\cinfty$-ring. Then the category $\modules{\cpt{\cring}}$ of $\cpt{\cring}$-modules in $\coubo$ is quasi-abelian and $\simpo{\modules{\cpt{\cring}}}\cong\modules{\cpt{\simo{\cring}}}$. We have a functor $\hom_{\coubo}(\samcou,-)\colon\modules{\cpt{\simo{\cring}}}\longrightarrow\simpo{\abel}$, and we denote by $\qcoh{\cring}\subseteq\pi_0(\modules{\cpt{\simo{\cring}}})$ the full subcategory\footnote{By $\pi_0(-)$ of a model category we denote the corresponding homotopy category.} consisting of $\simo{\module}$ s.t.\@ $\pi_{>0}(\hom_{\coubo}(\samcou,\simo{\module})=0$. We have the following statement.

\begin{proposition}\label{GoodLeftHeart} Let $\cring$ be a Fr\'echet $\cinfty$-ring of finite presentation. Then $\qcoh{\cring}$ is equivalent to a left abelian envelope of $\modules{\cpt{\cring}}$ (Def.\@ \ref{LeftAbelianEnvelope}).\end{proposition}
\begin{proof} According to Cor.\@ 1.2.20 and Prop.\@ 1.2.35 in \cite{Schneiders99} a left abelian envelope of $\modules{\cpt{\cring}}$ can be described as a localization of the following category: objects are monomorphisms $\module'\rightarrow\module$ in $\modules{\cpt{\cring}}$ and morphisms are commutative squares. The localization happens by inverting commutative squares that are simultaneously cartesian and co-cartesian. 

Let $\simo{\module'}$, $\simo{\module}\in\modules{\cpt{\simo{\cring}}}$ whose weak equivalence classes lie in $\qcoh{\cring}$. Let $\simo{\mu}\colon\simo{\module'}\rightarrow\simo{\module}$ be a weak equivalence. Switching first to $\simpo{\modules{\cpt{\cring}}}$ and then to $\compo{\modules{\cpt{\cring}}}$ and truncating we obtain
	\begin{equation}\label{CoCartesian}\xymatrix{\widetilde{\module}'_1\ar[d]\ar[rr] && \widetilde{\module}_1\ar[d]\\
	\module'_0\ar[rr] && \module_0,}\end{equation}
where $\widetilde{\module}'_1,\widetilde{\module}_1:=\kernel{\boun{0}}/(\kernel{\boun{0}\cap\kernel{\boun{1}}})$ in simplicial dimension $1$. The vertical arrows are given by $\boun{1}$ and hence are monomorphisms. 

Evaluating at $\samcou$ we see that $\pi_0(\simo{\mu})$ being injective implies (\ref{CoCartesian}) is cartesian in $\modules{\cpt{\cring}}$. \hide{%
Indeed, a morphism from $\samcou$ to $\module'_0\rightarrow\module_0\leftarrow\widetilde{\module}_1$ that does not factor through $\widetilde{\module}'_1$ represents a non-trivial element in $\pi_0(\simo{\module'})$ that becomes trivial in $\pi_0(\simo{\module})$.}%
Since $\pi_*(\simo{\mu})$ is an isomorphism, the dg $\cpt{\cring}$-module $\widetilde{\module}'_1\rightarrow\widetilde{\module}_1\oplus\module'_0\rightarrow\module_0$ (the cone of $\simo{\mu}$) is acyclic, i.e.\@ (\ref{CoCartesian}) is also co-cartesian. Conversely, suppose that by truncating a $\simo{\mu}$ as above we obtain (\ref{CoCartesian}) that is both cartesian and co-cartesian. The cartesian bit implies injectivity of $\pi_0(\simo{\mu})$ and the co-cartesian implies surjectivity.\end{proof}%

\begin{remark} Prop.\@ \ref{GoodLeftHeart} lets us view the left abelian envelope of $\modules{\cpt{\cring}}$ as a quotient of the category of cofibrant simplicial $\cpt{\simo{\cring}}$-modules with $\pi_{>0}=0$ by the homotopy relation on morphisms in $\modules{\cpt{\simo{\cring}}}$.\end{remark}

\begin{proposition} Let $\phi\colon\cring'\rightarrow\cring$ be a morphism of Fr\'echet $\cinfty$-rings of finite presentation. It induces an adjunction $\qcoh{\cring'}\rightleftarrows\qcoh{\cring}$.\end{proposition}
\begin{proof} We have the Quillen adjunction $\modules{\cpt{\simo{\cring'}}}\rightleftarrows\modules{\cpt{\simo{\cring}}}$, that descends to an adjunction $\pi_0(\modules{\cpt{\simo{\cring'}}})\rightleftarrows\pi_0(\modules{\cpt{\simo{\cring}}})$ The truncation functor $\simo{\module}\mapsto(\widetilde{\module}_1\rightarrow\module_0)$ on $\modules{\cpt{\simo{\cring}}}$ preserves weak equivalences and hence descends to $\pi_0(\modules{\cpt{\simo{\cring}}})$, where it is left adjoint to the inclusion. Clearly $\pi_0(\modules{\cpt{\simo{\cring'}}})\leftarrow\pi_0(\modules{\cpt{\simo{\cring}}})$ preserves the property $\pi_{>0}=0$.\end{proof}%

\medskip

If we cared only for Fr\'echet $\cinfty$-rings of finite presentation this would have been enough. We could then define the corresponding derived category of $\qcoh{\cring}$ and investigate the functoriality properties. However, we would also like to work with $\cinfty$-rings that are not Fr\'echet (e.g.\@ $\cinfty(\mathbb R)/(e^{-\frac{1}{x^2}})$) or not of finite presentation (e.g.\@ $\mathbb R[[x]]$). This forces us to use almost free resolutions $\resolution{\simo{\cring}}$ and the corresponding categories of modules.

The full $\modules{\cpt{\resolution{\simo{\cring}}}}$ is not a good candidate for the {\it derived} category of quasi-coherent sheaves. One of the reasons lies with coherent objects. As usual a coherent sheaf should be defined as an object homotopically of finite presentation. All Banach bundles are coherent in this way, as well as {\it all} of their quotients. Then also their sub-objects should be coherent (e.g.\@ $\cinfty(X)$ for a compact $X$). This means we need to stabilize $\modules{\cpt{\resolution{\simo{\cring}}}}$. 

\begin{definition} Let $\simo{\cring}$ be a simplicial $\cinfty$-ring. {\it The derived category of quasi-coherent sheaves} over $\simo{\cring}$ is the stabilization of $\modules{\cpt{\resolution{\simo{\cring}}}}$ with respect to the suspension functor.\end{definition}
In order to keep this work within a reasonable size we postpone the analysis of stabilization, extraction of coherent objects and the corresponding functorial properties to another paper. 

We note however that already with the techniques developed in the current work we can construct categories of quasi-coherent sheaves on arbitrary sheaves on the site of $\cinfty$-rings and Zariski topology. As usual this is done by the left Kan extension. 

In particular we can apply this to de Rham spaces and obtain categories of $D$-modules. Our insistence on allowing all, not only Fr\'echet $\cinfty$-rings, lets us construct $D$-modules for both kinds of nilpotents considered in \cite{BK2017}.


\begin{thebibliography}{Bour}
\bibitem{AdamekRosicky94} J.Ad\'amek, J.Rosicky. {\it Locally presentable and accessible categories.} Cambridge UP (1994)
\bibitem{SGA4.1} M.Artin, A.Grothendieck, J.L.Verdier. {\it S\'eminaire de g\'eometrie alg\'ebrique du Bois-Marie IV/1. Th\'eorie des topos et cohomologie \'etale des sch\'emas.} LNM 269 Springer (1972)
\bibitem{BambozziBenBassat16} F.Bambozzi, O.Ben-Bassat {\it Dagger geometry as Banach algebraic geometry.} Jour.\@ of Number Theory 162, pp.\@ 391-462 (2016)
\bibitem{BBK15} F.Bambozzi, O.Ben-Bassat, K.Kremnizer. {\it Stein domains in Banach algebraic geometry.} arXiv:1511.09045v1 [math.FA]
\bibitem{BD04} A.Beilinson, V.Drinfeld. {\it Chiral algebras.} AMS (2004)
\bibitem{BK15} O.Ben-Bassat, K.Kremnizer. {\it Non-Archimedean analytic geometry as relative algebraic geometry.}  arXiv:1312.0338
\bibitem{OrenKobi} O.Ben-Bassat, K.Kremnizer. {\it A perspective on the foundations of derived analytic geometry.} Preprint.
\bibitem{BK2017} D.Borisov, K.Kremnizer. {\it Beyond perturbation 1: de Rham spaces.} arXiv:1701.06278 [math.DG]
\bibitem{BourTopSpaces1-5} N.Bourbaki. {\it Topological vector spaces. Chapters 1-5.} Springer (1987)
\bibitem{ChristensenHovey02} J.D.Christensen, M.Hovey. {\it Quillen model structures for relative homological algebra.} Math.\@ Proc.\@ Camb.\@ Phil.\@ Soc.\@ 133, pp.\@ 261-293 (2002)
\bibitem{CostelloGwilliam} K.Costello, O.Gwilliam. {\it Factorization algebras in quantum field theory. Volume 1.} Cambridge UP (2016)
\bibitem{Du01} D.Dugger. {\it Combinatorial model categories have presentations.} Advances in Math.\@ 164, pp.\@ 177-201 (2001)
\bibitem{FB-Z04} E.Frenkel, D.Ben-Zvi. {\it Vertex algebras and algebraic curves. Second edition.} AMS (2004)
\bibitem{GJ99} P.Goerss, J.Jardine. {\it Simplicial homotopy theory.} Birkh\"{a}user Prog. in Mathematics 174, XVI+510 pages (1999)
\bibitem{GoerssSch04} P.Goerss, K.Schemmerhorn. {\it Model categories and simplicial methods.} Contemporary Mathematics 436 AMS (2004)
\bibitem{Grothendieck66} A.Grothendieck. {\it Produits tensoriels topologiques et espaces nucl\'eaires.} Memoirs of the AMS No.\@ 16 (1966)
\bibitem{GS03} J.A.N. Gonz\'{a}lez, J.B.S. de Salas. {\it $C^\infty$-differentiable spaces.} Lecture Notes in Mathematics 1824 Springer, XIII+188 pages (2003).
\bibitem{Hewitt43} E.Hewitt. {\it A problem of set-theoretic topology.} Duke Math.\@ J.\@ 10, pp.\@ 309-333 (1943)
\bibitem{Hogbe-Nlend1977} H.Hogbe-Nlend. {\it Bornologies and functional analysis.} North-Holland Mathematics Studies No.\@ 26, 144+XII pp.\@ (1977)
\bibitem{HM81} H.Hogbe-Nlend, V.B.Moscatelli. {\it Nuclear and conuclear spaces.} North-Holland (1981)
\bibitem{Ho99} M.Hovey. {\it Model categories.} AMS Mathematical surveys and monographs, Volume 63, XII+209 pages (1999)
\bibitem{Johnstone02} P.T.Johnstone. {\it Sketches of an elephant - A topos theory compendium.} Clarendon Press (2002)
\bibitem{KKM87} G.Kainz, A.Kriegl, P.Michor. {\it $C^\infty$-algebras from the functional analytic view point.}  J. of Pure and App. Algebra 46, pp. 89-107 (1987)
\bibitem{KM97} A.Kriegl, P.W.Michor. {\it The convenient setting of global analysis.} AMS (1997)
\bibitem{KashiwaraSchapira06} M.Kashiwara, P.Schapira. {\it Categories and sheaves.} Springer (2006)
\bibitem{Kelly16} J.Kelly. {\it Projective model structures for exact categories.} arXiv:1603.06557 
\bibitem{Koethe66} G.K\"othe. {\it Hebbare lokalkonvexe R\"aume.} Math.\@ Annalen 165, pp.\@ 181-195 (1966)
\bibitem{KomTo06} P.Komj\'ath, V.Totik. {\it Problems and theorems in classical set theory.} Springer (2006)
\bibitem{Marty09} F.Marty. {\it Des ouverts Zariski et des morphismes lisses en g\'eom\'etrie relative.} Th\'ese du doctorat de l'Universit\'e de Toulouse (2009)
\bibitem{Meyer2007} R.Meyer. {\it Local and analytic cyclic homology.} European Mathematical Society, 360+VIII pp. (2007)
\bibitem{Meyer08} R.Meyer. {\it Homological algebra for Schwartz algebras.} Pp.\@ 99-114 in {\it Symmetries in algebra and number theory.} Universit\"atsverlag G\"ottingen (2008)
\bibitem{MR91} I.Moerdijk, G.E.Reyes. {\it Models for smooth infinitesimal analysis.} Springer, X+399 pages (1991).
\bibitem{Prosmans95} F.Prosmans. {\it Alg\`ebre homologique quasi-ab\'elienne.} DEA thesis at Universit\'e Paris-Nord (1995)
\bibitem{ProsmansSchneiders00} F.Prosmans, J-P.Schneiders. {\it A homological study of bornological spaces.} Preprint LAGA 00-21, Universit\'e Paris 13, 46 pp.\@ (2000)
\bibitem{Qu67} D.Quillen. {\it Homotopical algebra.} Lecture Notes in Mathematics 43, Springer (1967)
\bibitem{SB72} {\it S\'eminaire Banach.} Edit\'e par C.Houzel, LNM Springer (1972)
\bibitem{Schneiders99} J-P.Schneiders {\it Quasi-abelian categories and sheaves.}  M\'em.\@ Soc.\@ Math.\@ Fr.\@ (N.S.) no.\@ 76 (1999)
\bibitem{SchwedeShipley00} S.Schwede, B.E.Shipley. {\it Algebras and modules in monoidal model categories.} Proc.\@ London Math.\@ Soc.\@ (3) 80, pp.\@ 491-511 (2000)
\bibitem{WN05} L.Waelbroeck, G.No\"el. {\it Bornological quotients.} Acad\'emie royale de Belgique (2005)
\bibitem{Wh48} H.Whitney. {\it On ideals of differentiable functions.} American Journal of Mathematics. Vol. 70, No. 3, pp. 635-658 (1948)
\end{thebibliography}
\end{document}